\documentclass{conm-p-l}

\usepackage{amssymb}
\usepackage{amsmath}
\usepackage{amsthm}
\usepackage{mathtools}
\usepackage{pxfonts}
\usepackage{algorithmic}
\usepackage{enumerate}
\usepackage{booktabs}
\usepackage{caption}
\usepackage{placeins}
\usepackage{xcolor}
\usepackage{color}
\usepackage{fullpage}
\usepackage{comment}
\definecolor{mylinkcolor}{rgb}{0.5,0.0,0.0}
\definecolor{myurlcolor}{rgb}{0.0,0.0,0.75}
\usepackage[colorlinks, urlcolor=myurlcolor, citecolor=myurlcolor, linkcolor=mylinkcolor,pagebackref,breaklinks=true]{hyperref}

\newtheorem{theorem}{Theorem}             
\newtheorem{lemma}[theorem]{Lemma}
\newtheorem{proposition}[theorem]{Proposition}
\newtheorem{corollary}[theorem]{Corollary}
\theoremstyle{remark}
\newtheorem{remark}[theorem]{Remark}
\theoremstyle{definition}
\newtheorem{algorithm}{Algorithm}

\numberwithin{equation}{section}

\newcommand{\defi}[1]{\emph{#1}}
\newcommand{\software}[1]{\textsc{#1}{}}
\newcommand{\Sage}{\software{SageMath}}
\newcommand{\Magma}{\software{Magma}}
\newcommand{\Github}{\software{GitHub}}
\newcommand{\GP}{\software{Pari/GP}}
\newcommand{\Arb}{\software{Arb}}
\newcommand{\true}{\texttt{true}}
\newcommand{\false}{\texttt{false}}

\newcommand{\Q}{\mathbb{Q}}
\newcommand{\Qbar}{\overline{\mathbb{Q}}}
\newcommand{\R}{\mathbb{R}}
\newcommand{\C}{\mathbb{C}}
\newcommand{\Z}{\mathbb{Z}}
\newcommand{\F}{\mathbb{F}}

\newcommand{\RR}{\mathcal{R}}
\newcommand{\OO}{\mathcal{O}}
\newcommand{\OD}{\mathcal{O}_D}
\newcommand{\OF}{\mathcal{O}_F}
\newcommand{\OK}{\mathcal{O}_K}

\newcommand{\ClD}{\mathcal{C}_D}
\newcommand{\Fp}{\mathbb{F}_p}

\newcommand{\Fq}{\mathbb{F}_\mathfrak{q}}
\newcommand{\Fl}{\mathbb{F}_{\ell}}
\newcommand{\Flbar}{\overline{\mathbb{F}}_\ell}
\newcommand{\Zl}{\mathbb{Z}_{\ell}}

\newcommand{\p}{\mathfrak{p}}
\newcommand{\q}{\mathfrak{q}}
\renewcommand{\a}{\mathfrak{a}}
\renewcommand{\b}{\mathfrak{b}}
\newcommand{\f}{\mathfrak{f}}
\renewcommand{\l}{\mathfrak{l}}
\renewcommand{\H}{|H|}
\newcommand{\HD}{|H_D|}

\newcommand{\leg}[2]{\left(\frac{#1}{#2}\right)}
\renewcommand{\ss}{{\rm ss}}
\newcommand{\pmin}{p_{\rm min}}

\newcommand{\arxiv}[1]{arXiv:\href{https://arxiv.org/abs/#1}{#1}}

\DeclareMathOperator{\GL}{GL}
\DeclareMathOperator{\SL}{SL}
\DeclareMathOperator{\disc}{disc}
\DeclareMathOperator{\ord}{ord}
\DeclareMathOperator{\Res}{Res}
\DeclareMathOperator{\Gal}{Gal}

\DeclareMathOperator{\End}{End}
\DeclareMathOperator{\Frob}{Frob}
\DeclareMathOperator{\round}{round}
\DeclareMathOperator{\llog}{llog}
\DeclareMathOperator{\Li}{Li}

\newcommand{\Frobp}{\Frob_p}
\newcommand{\Frobpp}{\Frob_{\p}}

\begin{document}

\title[Computing the endomorphism ring of an elliptic curve]{Computing the endomorphism ring\\of an elliptic curve over a number field}

\author{John E. Cremona}
\address{Mathematics Institute\\
University of Warwick\\
Coventry CV4 7AL\\
UK}
\email{J.E.Cremona@warwick.ac.uk}

\author{Andrew V. Sutherland}
\address{Massachusetts Institute of Technology\\
Department of Mathematics\\
77 Massachusetts Avenue\\
Cambridge, MA 02139\\
USA}
\email{drew@math.mit.edu}

\thanks{Sutherland was supported by Simons Foundation grant 550033.}

\date{19 April 2023}

\begin{abstract}
We describe deterministic and probabilistic algorithms to determine whether or not a given monic irreducible polynomial $H\in \Z[X]$ is a Hilbert class polynomial, and if so, which one.
These algorithms can be used to determine whether a given algebraic integer is the $j$-invariant of an elliptic curve with complex multiplication (CM), and if so, the associated CM discriminant.
More generally, given an elliptic curve~$E$ over a number field, one can use them to compute the endomorphism ring of $E$.
Our algorithms admit simple implementations that are asymptotically and practically faster than previous approaches.
\end{abstract}

\maketitle

\section{Introduction}

Let $E$ be an elliptic curve defined over a number field.
The (geometric) endomorphism ring $\End(E)$ is an arithmetic invariant that plays a key role in many theorems and conjectures, including those related to the distribution of Frobenius traces such as the Sato--Tate and Lang--Trotter conjectures, and those related to Galois representations associated to $E$, such as Serre's uniformity question.
It is known that the ring $\End(E)$ is isomorphic either to $\Z$, or to an order $\OO$ in an imaginary quadratic field, and in the latter case one says that $E$ has \defi{complex multiplication} (CM).
The endomorphism ring is a property of the $j$-invariant~$j(E)$, so one may equivalently ask whether a given algebraic number is a \defi{CM $j$-invariant}, historically called a \emph{singular modulus}.
The order $\OO$ associated to a CM $j$-invariant is uniquely determined by its discriminant~$D:=\disc(\OO)$, so a second question is to determine the negative integer~$D$.
These are the questions we address in this paper.

An alternative formulation of these questions may be expressed in terms of \defi{Hilbert class polynomials} (HCPs).
The HCP associated to a negative discriminant~$D$ is the irreducible polynomial~$H_D\in \Z[X]$ whose roots are the CM $j$-invariants of discriminant~$D$ (all of which are algebraic integers).
So we may instead ask whether a given monic, irreducible, integer polynomial $H$ is or is not an HCP, and if so, to determine the discriminant~$D$ for which $H=H_D$.

\vspace{4pt}

\noindent
\textbf{Problem 1:} Given an algebraic integer $j$ or its minimal polynomial $H\in \Z[X]$, determine whether $j$ is CM, equivalently, whether $H$ is an HCP, $H_D$.\vspace{4pt}

\noindent
\textbf{Problem 2:} If the answer to Problem 1 is ``yes'', determine the value of $D$.

\vspace{4pt}

These problems both have existing solutions that have been implemented in computer algebra systems such as \Magma{} \cite{Magma} and \Sage{} \cite{Sage}, but these implementations are not as efficient as they could be.  In this paper we present a nearly quasilinear time algorithm that is both easy to implement and substantially faster in practice than the methods previously available.\footnote{At the time of writing, \Magma{}~(version 2.27-5) and \Sage\ (version~9.7) were the current versions; however, \Sage\ version 10.0 (released 20 May 2023) uses our implementation of Algorithm~\ref{alg:two}. }

The degree of the Hilbert class polynomial $H_D$ is the class number $h(D)$.
For any positive integer $h$, only finitely many negative discriminants $D$ have class number $h(D)=h$.
For example, when $h=1$ there are 13 (of which 9 are fundamental), corresponding to the 13 CM $j$-invariants that lie in $\Q$.
For $h\le 100$ the complete list of discriminants $D<0$ with $h(D)=h$ is known.
All such fundamental discriminants $D_0$ were enumerated by Watkins~\cite{Watkins}, and Klaise \cite{Klaise} used this list and the formula for $h(D)$ as a multiple
of $h(D_0)$ (equation~\eqref{eq:hD} below) to determine all negative discriminants~$D$ with $h(D)\le 100$.
There are a total of 66,758 such discriminants, of which 42,272 are fundamental.

This allows one to address Problems 1 and 2 for algebraic integers of low degree using a simple lookup table, an approach that is eminently practical when the degree bound is small: for example, the $L$-functions and Modular Forms Database (LMFDB) \cite{LMFDB} currently contains 725,185 elliptic curves defined over number fields of degrees at most 6 whose CM discriminants can all be computed using a lookup table of 281 HCPs.
But this approach is impractical for larger degrees: over 2GB of memory is needed to store the HCPs of degree up to~100, and asymptotically one expects the space for $h(D)\le h$ to grow like $h^{4+o(1)}$, which is quasiquadratic in~$|D|$.

An answer to Problem 2 can be independently verified by computing $H_D$ for the claimed value of $D$ and verifying that $H=H_D$, which can be accomplished in $|D|^{1+o(1)}$ time \cite{BBEL,Enge,Sutherland:HCP,Streng}.  This is roughly quasilinear in the size of $H_D$, which has degree $h(D)=|D|^{1/2+o(1)}$ and logarithmic height $|H_D|=|D|^{1/2+o(1)}$. The goal of this paper is to solve Problems 1 and 2 just as quickly; we do not quite achieve this goal, but we come close.

Our strategy for Problem 1 is to search for a prime $\ell$ modulo which $H$ is squarefree and has a root in $\Flbar$ that is an ordinary $j$-invariant;
such primes are plentiful.
Once~$\ell$ is found, we compute the classical modular polynomial $\Phi_\ell\in \Z[X,Y]$ and use a resultant to determine if for every root $j$ of $H$ the polynomial $\Phi_\ell(X,j)$ has two roots; this will occur when they are CM $j$-invariants but not otherwise, as we prove in Propositions~\ref{cor:CM-ell-count} and \ref{prop:non-CM-mellj}.
To achieve an attractive running time, this algorithm requires an explicit bound on~$\ell$ when~$H$ is an HCP; for this we rely on the generalized Riemann hypothesis (GRH).
This yields our first result.

\begin{theorem}
Given a monic irreducible polynomial $H\in \Z[X]$, if the GRH is true then Algorithm~\ref{alg:one} returns \true{} if and only if $H$ is a Hilbert class polynomial.
It admits a deterministic implementation that runs in $(h^2\H)^{1+o(1)}$ time using $(h\H)^{1+o(1)}$ space.
\end{theorem}

Our strategy for Problem 2, which also solves Problem 1, is to search for a prime~$\ell$ modulo which~$H$ has an $\Fl$-rational root that is the $j$-invariant of an ordinary elliptic curve $E/\Fl$.
We then compute the discriminant $D$ of the endomorphism ring of $E$.
By computing $H_D$ and comparing it to $H$, we can solve Problems 1 and 2 simultaneously.
Under reasonable heuristic assumptions, including but not limited to the GRH, we obtain an algorithm that has a nearly quasilinear running time.

\begin{theorem}\label{thm:alg2complexityintro}
Given a monic irreducible polynomial $H\in \Z[X]$, Algorithm~\ref{alg:two} returns \true{} and the discriminant~$D$ if and only if $H$ is the Hilbert class polynomial $H_D$.
Under the heuristic assumptions of Theorem~\ref{thm:alg2complexity}, it can be implemented as a probabilistic algorithm of Las Vegas type that runs in $h(h+\H)^{1+o(1)}$ expected time using $h(h+\H)^{1+o(1)}$ space.
\end{theorem}

We should emphasize that, unlike Algorithm~\ref{alg:one}, the output of Algorithm~\ref{alg:two} is correct whether our heuristic assumptions (including the GRH) are true or not, it is only when analyzing the running time that our heuristics are used.  The best known upper bounds on $\H$ for HCPs imply $\H=h^{1+o(1)}$, a bound that appears to be essentially tight, and within the practical range the total size of HCPs appears to be $h\H^{1+o(1)}$; see Tables 3 and 4 in \cite{Sutherland:HCP}, for example.  This suggests that the running time of Algorithm 2 on inputs that are HCPs is quasilinear in the size of the input, although we are not able to prove this.

This result represents an improvement over previous algorithms described in the literature due to Achter~\cite{Achter}, Armitage~\cite{Armitage}, and Charles~\cite{Charles}, which we discuss in Section~\ref{sec:oldalgs}, where we also describe the algorithms previously implemented in \Magma{} and \Sage{}.
More significantly, the new algorithms are easy to implement and fast in practice.  Implementations of our algorithms in \Magma{}, \GP{}, and \Sage{} are available in our \Github{} repository \cite{repo}, and timings can be found in Table~1 in \S\ref{sec:timings}, along with timings for existing methods in \Magma{} and \Sage{}.

For the inputs $H_D$ of degree $h(D)= 200$ that we tested, all three implementations of Algorithm~\ref{alg:two} took less than two seconds to correctly identify $H_D$ as a Hilbert class polynomial and determine the value of~$D$, thereby solving Problems 1 and 2.  By contrast, the methods previously used by \Magma{} and \Sage{} took close to 20 minutes on the same inputs.  For degree 200 inputs of similar size that are \emph{not} Hilbert class polynomials (the case most likely to arise in practice), the difference is even more pronounced: Algorithm~\ref{alg:two} takes less than ten milliseconds, versus more than ten minutes for existing methods.

\section{Theoretical background}

To simplify matters, throughout this article we restrict our attention to fields whose characteristic is not $2$ or $3$.
All the facts we recall here have analogs that hold in characteristic 2 and 3, see \cite[Appendix A]{SilvermanI}, but we shall not have need of them (our algorithms can easily be implemented to never use such fields).
For similar reasons we exclude $j$-invariants $j=0,1728$ and the CM discriminant $D=-3,-4$ whenever it is convenient to do so.
Both of the Problems we consider are trivial to solve in these cases: the $j$-invariant $0$ is the root of the Hilbert class polynomial $H_{-3}(X)=X$ for the CM discriminant $-3$ of the maximal order of $\Q(\zeta_3)$, and the $j$-invariant $1728$ is the root of the Hilbert class polynomial $H_{-4}(X)=X-1728$ for the CM discriminant $-4$ of the maximal order of $\Q(i)$.

\subsection{Elliptic curves and their \texorpdfstring{$j$}{j}-invariants}

Recall that the \defi{$j$-invariant} of an elliptic curve~$E$ defined over a field~$F$ is an element $j(E)$ of $F$, and every element of $F$ is the $j$-invariant of an elliptic curve $E/F$.
Over a field whose characteristic is not~$2$ or~$3$ every elliptic curve $E/F$ is defined by an equation of the form $y^2=x^3+ax+b$, in which case its $j$-invariant is $j(E)\coloneqq 1728\cdot 4a^3/(4a^3+27b^2)$.
For any $j\in F$ we may construct $E/F$ with $j(E)=j$ by taking $E$ to by $y^2=x^3-1$ or $y^2=x^3-x$ when $j=0$ or $j=1728$, respectively, and otherwise defining $E$ as $y^2=x^3+ax+b$ with $a=-3j(j-1728)$ and $b=-2j(j-1728)^2$.

The $j$-invariant of $E$ uniquely determines its geometric isomorphism class: for elliptic curves $E_1/F$ and $E_2/F$ we have $j(E_1)=j(E_2)$ if and only if $E_1$ and $E_2$ are isomorphic over an algebraic closure of $F$.
Such an isomorphism is necessarily defined over a finite extension of $F$, and for $j$-invariants other than $0$ and $1728$ this extension is at most a quadratic extension: if $E_1$ and $E_2$ are not isomorphic over~$F$ then they are isomorphic over some quadratic extension $F(\sqrt{d})$ with $d\in \OF$, and if $E_1$ is defined by $y^2=x^3+ax+b$ then $y^2=x^3+d^2ax+d^3b$ is a defining equation for $E_2$ and we say that $E_2$ is the quadratic twist of $E_1$ by $d$.

Elliptic curves~$E$ over~$\C$ are isomorphic (both as elliptic curves and as complex analytic varieties) to $\C/\Lambda$ for some lattice~$\Lambda$ in~$\C$.
If $\Lambda$ has $\Z$-basis~$\omega_1,\omega_2$ with $\tau=\omega_1/\omega_2$ in the complex upper half-plane, then we set $j(\Lambda)=j(\tau)\coloneqq e^{-2\pi i \tau} + 744 + 196884\,e^{2\pi i \tau}+\cdots$, the value of the classical elliptic $j$-function
at~$\tau$, and we have $j(E)=j(\Lambda)$.
This value is independent of the (oriented) $\Z$-basis, and homothetic lattices have the same $j$-invariant.
When~$E$ is defined over a number field~$F$ of degree~$n$, with~$n$ embeddings~$\sigma_i:F\to\C$, the images of $j(E)\in F$ under the $\sigma_i$ are the $j$-invariants of the curves~$\sigma_i(E)/\C$.

\subsection{Basic CM facts} \label{subsec:basicCMfacts}

We refer the reader to \cite{Cox}, and also to \cite[Ch.~II]{SilvermanII} for CM by the maximal order of an imaginary quadratic field, and to \cite[Ch.~6]{Schertz} for the general case.

An \defi{imaginary quadratic order} is a finite index subring of the ring of integers $\OK$ of an imaginary quadratic field $K$.
Imaginary quadratic orders are in 1-to-1 correspondence with the set of \defi{imaginary quadratic discriminants}: negative integers $D$ that are squares modulo $4$.
Every such $D$ arises as the discriminant of a unique imaginary quadratic order $\OD$, and can be written as $D=f^2D_0$ where the \defi{fundamental discriminant} $D_0$ is equal to the discriminant $D_K$ of the imaginary quadratic field $K=\Q(\sqrt{D})=\Q(\sqrt{D_K})$ and $f\coloneqq [\OK:\OD]=\sqrt{D/D_K}$ is the \defi{conductor} of $\OD$.
The class number~$h(D)\coloneqq h(\OD)$ is the order of the \defi{class group} $\ClD$ of invertible fractional $\OD$-ideals modulo principal fractional $\OD$-deals.
The class numbers $h(D)$ and $h(\OK)$ are related by the formula \cite[Thm.~7.24]{Cox}
\begin{equation}\label{eq:hD}
h(D) = h(\OD) = \frac{h(\OK)f}{[\OK^\times:\OD^\times]}\prod_{p\,|\,f}\left(1-\left(\frac{D_K}{p}\right)\frac{1}{p}\right),
\end{equation}
where $[\OK^\times:\OD^\times]$ is $2$ (resp. 3) when $D_K=-4$ (resp. $D_K=-3$) and $f>1$, and $1$ otherwise, and $\left(\frac{D_K}{p}\right)\in\{0,\pm 1\}$ is a Kronecker symbol; the integer $h(\OK)$ divides $h(D)$.

Fix an embedding of $K$ in~$\C$.
The image of each invertible $\OD$-ideal $\a$ under this embedding is a lattice in~$\C$ homothetic to $\Z+\tau\Z$ for some $\tau\in K\subseteq\C$ that we may assume lies in the upper half plane.
We define $j(\a)$ to be the $j$-invariant of this lattice; such values $j(\a)$ are traditionally called \defi{singular moduli}.
Homothetic lattices have the same $j$-invariant, so $j(\a)$ depends only on the ideal class~$[\a]$ of~$\a$ in $\ClD$ and may be written as~$j([\a])$.
If $E$ is an elliptic curve defined over a number field $F$ with $\End(E)$ an order in the imaginary quadratic field $K$, then for every prime $\p$ of~$F$ at which $E$ has good reduction, lying above the rational prime~$p$, the reduction is ordinary if $\p$ splits in~$K$ and supersingular if $\p$ is inert in~$K$.  Moreover, in the ordinary case, the reduced curve~$\overline{E}$ has the same endomorphism ring~$\End(E)$.  (See \cite[Theorem 13.12]{Lang}.)

\subsection{Hilbert Class Polynomials and their properties}\label{subsec:HCPs}

The \defi{Hilbert class polynomial} $H_D(X)$ associated to the (not necessarily fundamental) imaginary quadratic discriminant~$D$ is the polynomial
\[
H_D(X) = \prod_{[\a]\in\ClD}(X-j([\a])).
\]
This monic polynomial of degree~$h(D)$ has coefficients in~$\Z$ and is irreducible over $K=\Q(\sqrt{D})$, see \cite[Theorem 6.1.2]{Schertz}.
Let $F$ be the number field of degree~$h(D)$ obtained by adjoining a root of~$H_D$ to~$\Q$.
The compositum $L\coloneqq KF$ is the \defi{ring class field} of the order~$\OD$.
It is a Galois extension of~$K$ of degree~$h(D)$, and the Galois group $A=\Gal(L/K)$ is isomorphic to~$\ClD$ via the Artin map; in particular, $A$ is abelian of order~$h(D)$.
When $D$ is a fundamental discriminant the ring class field of $\OD$ is the \defi{Hilbert class field} of $K$, its maximal unramified abelian extension, and in general the extension $L/K$ is ramified only at primes that divide the conductor~$f$ of $\OD$ and must be ramified at all odd primes that do.
The extension $L/\Q$ is also Galois, and unramified at primes that do not divide $D$.
The ring class field $L$ is the Galois closure of~$F/\Q$ except when $\ClD$ has exponent~$2$, in which case~$F$ is itself Galois over~$\Q$.
The Galois group $G=\Gal(L/\Q)$ is the extension of~$A$ by~$C_2=\Gal(K/\Q)$; more precisely, $G = A \rtimes C_2$, where~$C_2$ acts on~$A$ by inversion (see \cite[Lemma 9.3]{Cox}).

\begin{remark}
In the special case where $\ClD$ is an elementary $2$-group, this is a direct product, $A$ is normal in~$G$ which is Abelian, and $F/\Q$ is Galois; otherwise, $F/\Q$ has Galois closure~$L$.
We will refer to this special case as the \defi{abelian case}.
It only occurs finitely often, with $h(D)=2^e$ for $e\le4$.
Complete and explicit lists of these \emph{abelian discriminants}~$D$ can be found in Tables~3.7--3.14 of Voight's thesis~\cite{VoightThesis}:
there are 101 of them, consisting of $13$, $29$, $34$,
$21$, and~$4$ for $e=0$, $1$, $2$, $3$, and~$4$, respectively, with the maximal value of $|D|$ equal to $163$, $427$, $1435$, $3315$, and~$7392$, respectively.
\end{remark}

Let $h\coloneqq h(D)$ and let $h_2$ be the order of $A[2]\coloneqq \{a\in A\mid a^2=1\}$.
Then $h_2$ is a power of~$2$ that divides~$h$ and satisfies $h\equiv h_2\pmod2$, with $h_2=1$ if and only if $h$ is odd.
The conjugacy class $a\in A$ in~$G$ is~$\{a,a^{-1}\}$, so $a\in A$ is central in~$G$ if and only if $a\in A[2]$.

The set~$\RR$ of roots of $H_D(X)$ in~$\C$ consists of the values $j(\a)$ as $\a$ runs over ideals representing the ideal classes of~$\OD$.
For each $\a$ we have $F\simeq\Q(j(\a))$ and $L=K(j(\a))$.
For our algorithms we require detailed knowledge of the Galois action on these roots, including the action of complex conjugation and of Frobenius automorphisms; here we again refer to~\cite[Theorem 6.1.2]{Schertz} for the facts we need.

Via the Artin isomorphism $\ClD\simeq A$, the action of $[\a]\in \ClD$ on $j(\b)\in \RR$ is given by
\[
\a\colon j(\b)\mapsto j(\a^{-1}\b) = j(\overline{\a}\b),
\]
and in the nonabelian case the $h(D)$ elements of $G\simeq \ClD \rtimes C_2$ not in $\ClD$ act via
\[
j(\b)\mapsto j(\a\overline{\b}).
\]

\subsubsection*{Action of complex conjugation}

We have
\[
\overline{j(\a)} = j(\overline{\a}) = j(\a^{-1});
\]
thus $j(\a)$ is real if and only if $\a^2$ is principal.
It follows that $h_2$ roots of $H_D$ are real and the rest form $(h-h_2)/2$ pairs of complex conjugates.
We always have at least one real root $j_0\coloneqq j(\OD)$, and the abelian case occurs if and only if all roots are real.

\subsubsection*{Action of Frobenius automorphisms}

Assume for the moment that $D$ is not one of the 101 abelian discriminants for which $\ClD$ is an elementary abelian 2-group.
Let $p$ be a prime such that $H_D(X)$ is squarefree modulo~$p$.
Then $p$ is unramified in~$L=KF$ (hence unramified in both~$K$ and~$F$), and there is a well-defined conjugacy class~$\Frobp\subseteq G$.
The action of~$\Frobp$ on~$\RR$ determines how the reduction of $H_D$ modulo $p$ factors in $\Fp[X]$: there is an irreducible factor of degree $d$ for each orbit of size $d$ (note that the orbit sizes are an invariant of the conjugacy class $\Frobp$).
We have $\Frob_p\subseteq A$ if and only if the prime $p$ splits in $K$.

Let $\p$ be a prime of~$K$ lying above~$p$ and let $\p_D\coloneqq \p \cap \OD$ denote the corresponding $\OD$-ideal, which is invertible because $\p$ is unramified in the ring class field $L$ and therefore prime to the conductor of $\OD$.
Since $A=\Gal(L/K)$ is abelian, there is a well-defined automorphism~$\Frobpp\in A$ which acts on~$\RR$ via
\[
j(\b)^{\Frobpp} = j(\p_D^{-1}\b).
\]
The order of~$\Frobpp\in A$ is the order~$n$ of the ideal class~$[\p_D]\in\ClD$, a divisor of~$h(D)$, and in fact a divisor of the exponent of~$\ClD$.
As a permutation of the roots, $\Frobpp$ acts as a product of~$h(D)/n$ cycles of length~$n$.
Thus~$H_D(X)$ factors over the residue field $\F_\p\coloneqq \OK/\p$ as a product of~$h(D)/n$ irreducible polynomials in $\F_\p[X]$ of degree~$n$.

When~$p$ splits in~$K$, say $p\OK=\p\overline{\p}$, the conjugacy class $\Frobp\subseteq G$ consists of two elements, $\Frobpp$ and~$\Frob_{\overline{\p}}=\Frobpp^{-1}$, each of order~$n$ as above.
We have $\OK/\p\simeq\Fp$, so the reduction of $H_D$ modulo $p$ factors in $\Fp[X]$ as a product of $h(D)/n$ irreducible polynomials of degree~$n$.
In particular, $H_D$ splits completely in~$\Fp[X]$ if and only if~$\p$ is principal, in which case~$p$ also splits completely in~$L$.

On the other hand, if~$p$ is inert in~$K$ then $\p=p\OK$ is principal and $\OK/\p\simeq\F_{p^2}$.
In this case $n=1$, $\Frobpp$ is trivial, and $H_D(X)$ splits completely in $\F_{p^2}[X]$; its factorization in~$\Fp[X]$ is a product of linear or irreducible quadratic factors.
Now $\Frobpp$ acts on the roots of $H_D(X)$ via $j(\b)\mapsto j(\p_D\overline{\b})$.
The number of linear factors is the number of fixed points of this action, which is either~$h_2$ or~$0$, the number of solutions to~$[\b]^2=[\p_D]$.
Thus for inert primes~$p$, the factorization of~$H_D(X)$ in~$\Fp[X]$ is \emph{either} a product of $h(D)/2$ quadratic factors (when $[\p_D]$ is not a square),
\emph{or} a product of~$h_2$ linear and $(h(D)-h_2)/2$ quadratic factors (when $[\p_D]$ is a square).
Note that the second possible factorization pattern in this case is the same as for the factorization of $H_D(X)$ in~$\R[X]$.
If $h(D)$ is odd, then $h_2=1$, and only the latter possibility occurs.

If $D$ is one of the 101 abelian discriminants then the fields $K$ and $F$ are linearly disjoint Galois extensions of $\Q$ and the action of $\Gal(F/\Q)\simeq \Gal(L/K)= A$ does not depend on how $p$ splits in $K$.
In this situation $A$ is an elementary 2-group, $h(D)=h_2$ is a power of $2$, and the reduction of $H_D$ modulo $p$ either splits into $h_2$ linear factors in $\Fp[x]$, or it is a product of $h_2/2$ irreducible quadratic factors in $\Fp[X]$.

\subsection{Galois theoretic constraints}

Using the explicit description of the action of $\Gal(L/\Q)$ and $\Gal(L/K)$ on the set~$\RR$ of roots of $H_D(X)$, we can formulate conditions on the factorization of a given monic irreducible polynomial $H\in \Z[X]$ in $\R[X]$ and $\Fp[X]$ that must be satisfied if $H$ is an HCP.

As noted above, primes $p$ that do not divide $D$ are unramified in $L$, but they may divide the discriminant of the polynomial $H_D(X)$, in which case $H_D$ will not be squarefree in $\Fp[X]$.
However, this cannot happen if the prime $p$ splits in $K$.

\begin{proposition}\label{prop:squarefree}
Let $D$ be an imaginary quadratic discriminant and let $p$ be a prime for which $\left(\frac{D}{p}\right)=+1$.
Then the Hilbert class polynomial $H_D$ is squarefree modulo $p$.
\end{proposition}
\begin{proof}
Let $K=\Q(\sqrt{D})$, let $p\OK=\p\bar\p$, and let $\q$ be a prime of $L$ above $\p$.
As noted above, $p$ is unramified in $L$, so $\q$ is unramified and its residue field $\Fq\coloneqq\mathcal{O}_L/\q$ is an extension of $\Fp=\F_\p=\OK/\p$.
Since the roots of~$H_D$ all lie in~$\OO_L$, it splits into linear factors over~$\Fq$, and we must show that these are distinct, or equivalently that $j([\a_1])\not=j([\a_2])$ implies $j([\a_1])\not\equiv j([\a_2])\pmod{\q}$.
We use the fact (see the end of subsection~\ref{subsec:basicCMfacts}) that every elliptic curve~$E$ with CM by~$\OD$ has good ordinary reduction modulo~$\q$, and that the reduction~$\overline{E}$ has the same endomorphism ring, $\OD$.
Given distinct $j_1=j([\a_1])$ and~$j_2=j([\a_2])$, we have $[\a_1]=[\a_2][\l]$, where~$\l$ is a nonprincipal ideal, and we may take~$\l$ to be a prime ideal, since the ideal class~$[\a_1\a_2^{-1}]$ contains infinitely many prime ideals.
Let~$\ell=N(\l)$; since~$\l$ is not principal, $\ell$ is prime and~$\OD$ has no elements of norm~$\ell$.
Let $E_1$, $E_2$ be elliptic curves defined over~$L$ with $j$-invariants~$j_1$ and~$j_2$ respectively; then there is an~$\ell$-isogeny $E_1\to E_2$, which reduces to an~$\ell$-isogeny between their reductions, $\overline{E_1}$ and~$\overline{E_2}$.
As these have endomorphism ring~$\OD$, they have no endomorphisms of degree~$\ell$, so $\overline{E_1}$ and~$\overline{E_2}$ and not isomorphic, even over the algebraic closure of~$\Fq$; thus $j([\a_1])\not\equiv j([\a_2])\pmod{\q}$ as required.
\end{proof}

\begin{remark}
As pointed out by a referee, an alternative proof is possible using the theory of canonical lifts for curves with ordinary reduction.
\end{remark}

\subsubsection*{Factorization over~$\R$}\label{subsec:real}

Recall that we have fixed an embedding $L\hookrightarrow \C$, so we can identify complex conjugation~$c$ as an element of~$G$, and $G\setminus A = \{ac\mid a\in A\}$.
The set of roots is $\{a(j_0)\mid a\in A\}$, and the subset of real roots is $\{a(j_0)\mid a\in A[2]\}$, with cardinality~$h_2$.
This gives the first necessary condition for $H$ to be an HCP.

\begin{proposition} \label{prop:real-roots}
Let $H(X)\in\Z[X]$ be a polynomial of degree~$h$ with exactly $h^+$ real roots.
If $H$ is a Hilbert class polynomial then the following hold:
\begin{enumerate}[(1)]
\item $h^+\mid h$;
\item $h^+$ is a power of~$2$;
\item $h^+\equiv h\pmod{2}$; that is, $h^+=1$ if and only if $h$ is odd.
\end{enumerate}
\end{proposition}

For example, an HCP of odd degree has exactly one real root, so an odd degree polynomial with any other number of real roots is not an HCP.

\subsubsection*{Application to CM testing}

Given a monic irreducible $H\in\Z[X]$
of degree~$h$, compute the number of real roots~$h^+$.
If any of the three conditions in Proposition~\ref{prop:real-roots} does not hold then $H$ is not an HCP.

\begin{remark}
The function {\tt CMtest} in \Magma{}~(version 2.27-5) determines the number of real roots and checks that it is a power of~$2$, but the other conditions of Proposition~\ref{prop:real-roots} are not checked.
See \S\ref{subsec:magma} below for more details.
\end{remark}

\begin{remark}
In practice, finding the real roots of an integer polynomial of large degree with large coefficients can be very time consuming.
Even just determining the number of real roots (using Sturm sequences, for example) can be expensive.
Our algorithms do not use this criterion.
\end{remark}

\subsubsection*{Factorization over~$\Fp$}

Let $p$ be a prime such that $H_D$ is squarefree modulo~$p$.
Our earlier discussion shows that the factorization of $H_D$ in $\Fp[X]$ must be one of the following, depending on whether $p$ is split or inert in $K$, or if $D$ is one of the 101 abelian discriminants.
As above, we put $h=h(D)=\deg(H_D)$ and $h_2=\#\ClD[2]$.
\begin{itemize}
\item (split) $h(D)/n$ irreducible factors of degree~$n$, for some $n\mid h$;
\item (inert) $h_2$ factors of degree $1$, and $(h-h_2)/2$ irreducible factors of degree~$2$;
\item (abelian) either $h_2$ factors of degree $1$, or $h_2/2$ factors of degree $2$.
\end{itemize}
Since the number of linear factors $h_2$ is equal to the number of real roots $h^+$ of $H_D$ it must satisfy conditions (1)--(3) of Proposition~\ref{prop:real-roots}.
Note that in the abelian case $\ClD$ is an elementary 2-group and $h=h_2$.
In this situation the factorization of $H$ in $\Fp[X]$ does not determine how $p$ splits in $K$.

Recall that the \emph{supersingular polynomial} in characteristic~$p$ is the monic polynomial $\ss_p\in\Fp[X]$ whose roots are the supersingular $j$-invariants in characteristic~$p$.

\begin{proposition} \label{prop:mod-p-factors}
Let $H$ be a Hilbert class polynomial with $h^+$ real roots, and let $p$ be a prime for which the reduction $\overline H$ of $H$ modulo $p$ is squarefree in $\Fp[X]$.
Among the irreducible factors of $\overline H$
\begin{itemize}
\item[either] all irreducible factors have the same degree
\item[or] there are $h^+$ linear factors and all others factors have degree~$2$.
\end{itemize}
Moreover, if both linear and nonlinear factors occur, then $\overline H$ divides $\ss_p$.
\end{proposition}
\begin{proof}
Except for the last sentence, the proposition follows from the discussion above.
For the last sentence, note that if both linear and nonlinear factors occur then the nonlinear factors must have degree 2 and $p$ must be inert in~$K$, so that the roots of~$\overline{H}$ are supersingular.
\end{proof}

\subsubsection*{Application to CM testing}

Given monic irreducible $H\in\Z[X]$ of degree~$h$, consider the factorization of the reduction $\overline H$ of~$H$ modulo~$p$ in $\Fp[X]$ for successive primes~$p$, omitting those for which $\overline H$ is not squarefree.
If the irreducible factors do not all have the same degree and the maximum degree is greater than~$2$, then $H$ is not an HCP.
Otherwise, if degrees~$1$ and~$2$ occur but the number of linear factors does not satisfy conditions (1)--(3) of Proposition~\ref{prop:real-roots}, then $H$ is not an HCP.

\begin{remark}
  For a detailed account of the factorization of Hilbert Class Polynomials over prime fields, see the preprint \cite{LiLiOuyang}.
\end{remark}

\subsection{Galois representation constraints}\label{subsec:GalRep}

Let $j\ne 0,1728$ be an algebraic integer.
Let $E_j$ be the elliptic curve over $F=\Q(j)$ with $j(E_j)=j$ defined by the equation
\[
y^2 = x^3 -3 j (j-1728) x -2 j (j-1728)^2
\]
with discriminant $\Delta_j\coloneqq 2^{12} 3^{6} j^2 (j - 1728)^3$, so that $E_j$ has good reduction at primes not dividing $6j(j-1728)$.
Let $G_F\coloneqq \Gal(\overline{F}/F)$ denote the absolute Galois
group of~$F$. For each rational prime $\ell$, we have the mod-$\ell$ Galois
representation~$\overline{\rho}_{E_j,\ell}\colon G_F\to\GL_2(\Fl)$
given by the action of $G_F$ on $E_j[\ell]$ (after fixing an $\Fl$-basis).

Suppose $E_j$ has CM by $\OD$ for $D=D_0f^2$, and let $K=\Q(\sqrt{D})=\Q(\sqrt{D_0})$.
Then the image of~$\overline{\rho}_{E,\ell}$ in~$\GL_2(\Fl)$ lies in either the normalizer of a split Cartan subgroup, the normalizer of a non-split Cartan subgroup, or a Borel subgroup of~$\GL_2(\Fl)$, depending on whether $\ell$ is split, inert or ramified in~$K$, respectively.

Let $\p$ be a prime of~$F$ not dividing~$\Delta_j$ so that $E_j$ has good reduction $\overline{E}_j$ at~$\p$.
Let $\F_\p\coloneqq \OO_F/\p$ denote the residue field of $\p$, let $a_{\p}\coloneqq N(\p)+1-\#\overline{E}_j(\F_\p)$ denote the trace of Frobenius,
and consider the quadratic polynomial $f_{\p}(X)=X^2-a_{\p}X+N(\p) \in\Z[X]$.
If $\p\nmid\ell$, then $\overline{\rho}_{E_j,\ell}$ is unramified at~$\p$, and $f_{\p}(X)$ is the characteristic polynomial of $M_{\p,\ell} = \rho_{E_j,\ell}(\Frob\p)\in\GL_2(\Zl)$.
Let $\delta_{\p} = a_{\p}^2-4N(\p) \in \Z$ be the discriminant of~$f_{\p}(X)$, and let $\overline{M}_{\p,\ell}\in\GL_2(\Fl)$ denote the reduction of~$M_{\p,\ell}$ modulo~$\p$.

In both the split and non-split cases, $\overline{M}_{\p,\ell}$ lies in a Cartan subgroup (rather than its normalizer) if and only if $\p$ splits in $KF/F$ (by the argument used in the proof of \cite[Lemma~7.2]{Zywina}, for example), which occurs if and only if $D$ is a square modulo~$\p$; this will certainly be the case if the residue field $\F_\p$ has even degree over its prime field $\Fp$, since it will then contain $\F_{p^2}$.
Otherwise (when $D$ is not a square modulo~$\p$), $\overline{M}_{\p,\ell}$ lies in the nontrivial coset of the Cartan in its normalizer, in which case its trace satisfies $a_{\p}\equiv0\pmod{\ell}$.
Note that we can also have $a_{\p}\equiv0\pmod{\ell}$ when $\overline{M}_{\p,\ell}$ is in the Cartan subgroup.

Define a prime~$\p$ of~$F$ to be \emph{revealing} if
\begin{itemize}
\item $\p\nmid \Delta_j$; and
\item $a_{\p}^2 \notin \{0,4N(\p)\}$.
\end{itemize}

\begin{proposition}
With notation as above, the squarefree part of $\delta_{\p}$ equals the squarefree part of~$D$ for all revealing primes~$\p$, so that $K=\Q(\sqrt{D})=\Q(\sqrt{\delta_{p}})$.
In particular, the squarefree part of $\delta_{\p}$ is independent of~$\p$, for all revealing primes~$\p$ of~$K$.
\end{proposition}
\begin{proof}
Let $\p$ be a revealing prime.
Let $\ell$ be a rational prime such that $\ell\nmid D_0a_{\p}\delta_{\p}$ and $\p\nmid\ell$; these conditions exclude only finitely many primes $\ell$.
Since $\ell\nmid D_0$, the image $\overline{\rho}_{E_j,\ell}(G_F)$ is the normalizer of a Cartan subgroup.
The matrix $\overline{M}_{\p,\ell}$ lies in the Cartan subgroup itself since its trace $a_{\p}\not\equiv0\pmod{\ell}$, and since $\delta_{\p}\not\equiv0\pmod{\ell}$, it has distinct eigenvalues.

In the split case when $\leg{D}{\ell}=+1$, the eigenvalues of~$\overline{M}_{\p,\ell}$ lie in $\F_{\ell}$, so $\leg{\delta_{\p}}{\ell}=+1$.
In the nonsplit case, $\leg{D}{\ell}=-1$ and $\overline{M}_{\p,\ell}$ has conjugate eigenvalues in $\F_{\ell^2}$, so $\leg{\delta_{\p}}{\ell}=-1$.
Hence $\leg{\delta_{\p}}{\ell}=\leg{D}{\ell}$ for all but finitely many primes~$\ell$, and the result follows.
\end{proof}

\subsubsection*{Application to CM testing}

Given a monic irreducible $H\in\Z[X]$ of degree~$h$, let $F$ be the number field it defines, and let~$j$ be a root of~$H$ in~$F$.
Assuming that $j\not=0,1728$, let $E_j$ be the elliptic curve over~$F$ defined above, and compute $\Delta_j=2^{12} 3^{6} j^2 (j - 1728)^3$.
Initialize $d=0$.
Loop through primes~$\p$ of~$F$ not dividing~$\Delta_j$, compute the trace of Frobenius~$a_{\p}$ of~$E_j$ at~$\p$ and $\delta_{\p}=a_{\p}^2-4N(\p)$.
Skip~$\p$ if either~$a_{\p}=0$ or~$\delta_{\p}=0$.
If $d=0$ set $d=\delta_{\p}$, otherwise test whether $\delta_{\p}d$ is a square.
If it is not then $H$ is not an HCP.
Otherwise we may proceed to the next prime~$\p$.
If after checking many primes $\p$ we have failed to prove that $H$ is not an HCP then it likely is an HCP whose discriminant has the same squarefree part as~$d$.

\subsubsection*{Refinement}

Whenever $\delta_{\p}d$ is a square, replace $d$ by $\gcd(d,\delta_{\p})$ before continuing.
At the end, write $d=D_0f_1^2$ with~$D_0$ a fundamental discriminant.
Then $H$ is likely to be an HCP, and if it is, its discriminant must have the form $D=f^2D_0$ with $f\mid f_1$.
Compute $h(f_1^2D_0)$; if it does not divide $h=\deg(H)$ then $H$ is not an HCP.
If it does, we could then finish the test decisively by first computing $h(f^2D_0)$ for all such~$f$, testing whether~$h=h(f^2D_0)$, and if so, computing $H_{f^2D_0}$ and comparing with~$H$.\\

To justify the refinement, note that for each revealing prime~$\p$, the reduction of $E_j$ modulo~$\p$ is an ordinary elliptic curve $\overline{E}_j$, whose endomorphism ring is the same as that of~$E_j$, and whose Frobenius has discriminant~$\delta_{\p}$.
Hence the endomorphism ring of $\overline{E}_j$, which contains this Frobenius, has discriminant dividing~$\delta_{\p}$, with a square cofactor.
This holds for all the revealing primes used, so if $E_j$ is indeed a CM curve then $\End(E_j)$ is a quadratic order whose discriminant divides all these~$\delta_{\p}$ and hence divides their~$\gcd$, with square cofactors.
Here we have used the fact that $\End(\overline{E}_j)=\End(E_j)$ when $E_j$ is a CM elliptic curve and $\overline{E}_j$ is its reduction at a good ordinary prime.

\begin{remark}
Although our algorithm does not use this method, we include it here as it was the default method used in \Sage\ (version~9.7): see section~\ref{subsec:sage} below.
\end{remark}

\subsection{Modular polynomials}

Elliptic curves over~$\Qbar$ with CM have the property that they are isogenous (over~$\Qbar$) to all their Galois conjugates; that is, they are $\Q$-curves.
This can also happen for non-CM elliptic curves, but it is relatively rare for a non-CM elliptic curve to be isogenous to any of its Galois conjugates and it is impossible for a non-CM elliptic curve to admit two cyclic isogenies of the same degree with distinct kernels that lead to Galois conjugates, since this will force a cycle of isogenies among the Galois conjugates and a non-integer endomorphism.
We can use this dichotomy as the basis for a CM test.

It is convenient to define two algebraic numbers $j,j'$ to be \emph{isogenous} if elliptic curves $E,E'$ with these $j$-invariants are isogenous over~$\Qbar$.
This is an equivalence relation on~$\Qbar$, and if $j\in\Qbar$ is CM then so are all isogenous algebraic numbers.
Given~$j$ and a positive integer~$n$, the $j'$ which are $n$-isogenous to~$j$ (that is, isogenous via a cyclic $n$-isogeny) are the roots of the polynomial~$\Phi_n(X,j)$, where~$\Phi_n(X,Y)\in\Z[X,Y]$ is the $n$th classical modular polynomial.
When~$\ell$ is prime, $\Phi_{\ell}(X,j)$ has degree~$\ell+1$, and when $j$ is non-CM its roots are distinct.
Note that $j$ itself can be a root of~$\Phi_{\ell}(X,j)$ if $j$ is CM and the associated order has a principal ideal of norm~$\ell$.
We are interested in the roots of $\Phi_\ell(X,j)$ which are Galois conjugates of~$j$, possibly including~$j$ itself.

Let $m_{\ell}(j)$ denote the number of Galois conjugates of~$j$ that are roots of~$\Phi_{\ell}(X,j)$, counting multiplicities.
Then
\begin{equation} \label{eqn:res-formula}
m_{\ell}(j) = \ord_{H}(\Res_Y(H(Y),\Phi_{\ell}(X,Y))),
\end{equation}
where $H$ is the minimal polynomial of~$j$, and $\ord_H(F)$ is the
exponent of~$H$ in the factorization of~$F\in\Z[X]$.  The value
of~$m_{\ell}(j)$ is at most~$2$, as we will deduce from the following
result giving the complete factorization
of~$\Res_Y(H(Y),\Phi_{\ell}(X,Y))$, which has degree
$(\ell+1)\deg(H)$.

\begin{proposition} \label{prop:LuminyFormula}
Let $D=f^2D_0$ be a CM discriminant and let $\ell$ be prime.  Set
$e=\ord_{\ell}(f)\ge0$,  $\varepsilon =
\left(\frac{D_0}{\ell}\right)\in\{0,\pm 1\}$, and $w(D)=\frac{1}{2}\#\OO_D^{\times}$.  Then
\[
\Res_Y(H_D(Y),\Phi_{\ell}(X,Y)) = H_{\ell^{2}D}(X)^{w(D)} \cdot
\begin{cases}
H_D(X)^{1+\varepsilon} & \text{if $e=0$;} \\
H_{D/\ell^2}(X)^{(\ell-\varepsilon)/w(D/\ell^2)} & \text{if $e=1$;} \\
H_{D/\ell^2}(X)^{\ell} & \text{if $e\ge2$.}
\end{cases}
\]
\end{proposition}

\begin{proof}
We exploit the theory of isogeny volcanoes; see
\cite{Sutherland:volcano} for a summary of the facts we need, most of
which were proved by David Kohel in his thesis~\cite{Kohel}.  If $j$
is a root of~$H_D$, then the $\ell+1$ roots of $\Phi_{\ell}(X,j)$ are
the $j'$ (counting multiplicities) such that there exist
$\ell$-isogenies from $j'$ to~$j$; these have
discriminant~$D'\in\{\ell^2D, D, D/\ell^2\}$.  Note that $j'$ is
Galois conjugate to~$j$ if and only if $D=D'$, which only occurs when $e=0$.

If $e=0$, then $j$ is on the rim of the $\ell$-isogeny volcano; in
this case $\varepsilon+1$ of the $\ell$-isogenous $j'$ have the same
discriminant~$D$ as~$j$, while the other $\ell+1 - (\varepsilon+1) =
\ell-\varepsilon$ occur with multiplicity~$w(D)$ and have discriminant~$\ell^{2}D$.
This establishes the first line of the result.

If $e\ge1$, there are no $\ell$-isogenous $j'$ with discriminant~$D$.
Those with discriminant $D/\ell^2$ occur with multiplicity
$(\ell-\varepsilon)/w(D/\ell^2)$ when $e=1$ (by the $e=0$ case
applied to~$j'$), and with multiplicity~$\ell$ when $e\ge2$.
Those with discriminant $\ell^2D$ occur with multiplicity~$w(D)$.
\end{proof}

\begin{corollary} \label{cor:CM-ell-count}
Let $j$ be a CM $j$-invariant of discriminant $D=D_0f^2$, let $K=\Q(\sqrt{D})$, and let $\ell$ be prime.
Then $m_{\ell}(j) \le 2$ and one of the following holds:
\begin{itemize}
\item $m_{\ell}(j) = 0$ and either $\ell\mid f$, or $\ell\nmid D$ is inert in~$K$;
\item $m_{\ell}(j) = 1$ and $\ell\nmid f$ is ramified in $K$;
\item $m_{\ell}(j) = 2$ and $\ell\nmid D$ splits in~$K$.
\end{itemize}
\end{corollary}

\begin{proof}
  In the notation of the Proposition, $m_{\ell}(j)$ is the exponent
  of~$H_D(X)$ in the formula.  If $\ell\mid f$ then $e\ge1$, so
  $m_{\ell}(j) = 0$, while if $\ell\nmid f$ then $e=0$ and
  $m_{\ell}(j) = 1+\varepsilon$.  Lastly, $1+\varepsilon=0$, $1$,
  or~$2$ according to whether $\ell$ is inert, ramified or split
  in~$K$ respectively.
\end{proof}

\begin{corollary}
If $j$ is a CM $j$-invariant, the densities of primes~$\ell$ for which $m_{\ell}(j)=0$, $1$, $2$ are $1/2$, $0$, $1/2$, respectively.
\end{corollary}

We now determine the possible multiplicities~$m_{\ell}(j)$ in the case where $j$ is \emph{not} CM, and will see that it is usually~$0$, may be~$1$ for finitely many~$\ell$, but is never larger than 1.
Assume that $j\in\Qbar$ is non-CM.
The set $G_j$ of $\sigma\in\Gal(\Qbar/\Q)$ such that $\sigma(j)$ is isogenous to $j$ is a subgroup containing $\Gal(\Qbar/\Q(j))$; letting $F$ denote its fixed field, we see that $F$ is a subfield of~$\Q(j)$.
For most values of $j$ we will have $F=\Q(j)$, meaning that $j$ is not isogenous to any of its conjugates (except itself, trivially).
If $F=\Q$ then (by definition), $j$ is a $\Q$-number, in the sense that elliptic curves with $j$-invariant~$j$ are $\Q$-curves; in general $j$ is a $K$-number (with an analogous definition): see \cite{Elkies} and \cite{CremonaNajman}.
By the theory of $K$-curves, if there are any isogenies between $j$ and its conjugates of degree divisible by~$\ell$, they all factor through a unique $\ell$-isogeny, so that $m_{\ell}(j)=1$; otherwise $m_{\ell}(j)=0$.
This establishes the following proposition.

\begin{proposition}\label{prop:non-CM-mellj}
Let $j$ be an algebraic number that is not CM.
Then $m_{\ell}(j)=0$ for all but finitely many (and possibly no) primes~$\ell$, for which $m_{\ell}(j)=1$.
\end{proposition}

\begin{corollary}
If $j$ is non-CM then the densities of primes~$\ell$ for which $m_{\ell}(j)=0$, $1$ are $1$, $0$, respectively.
\end{corollary}

\subsubsection*{Application to CM testing}

Given an algebraic integer~$j$ with minimal polynomial~$H\in\Z[X]$, use 
equation~\eqref{eqn:res-formula} to compute $m_{\ell}(j)$ for a suitably chosen prime~$\ell$.
If $m_{\ell}(j)=2$ then $j$ is CM.
Moreover, if the reduction of $H$ modulo~$\ell$ is squarefree in $\Fl[X]$ with roots in $\F_{\ell^2}$ which are not supersingular $j$-invariants, then $m_{\ell}(j)=2$ if and only if $j$ is CM.
This observation is the basis of our first algorithm.

\subsection{Explicit bounds}\label{subsec:bounds}

In this section we recall some number-theoretic bounds needed for the analysis of our algorithms.

\begin{proposition}\label{prop:ordprimes}
Let $E$ be an elliptic curve defined over a number field $F$ of degree $n$.
Let~$S$ be the set of rational primes $p$ for which $F$ has a degree one prime $\p|p$ and every $\p|p$ is a prime of good ordinary reduction for $E$.
The set $S$ has density at least $1/(2n)$.
\end{proposition}
\begin{proof}
Let $L$ be the Galois closure of $F$ over $\Q$.
The Galois group $G\coloneqq \Gal(L/\Q)$ is a transitive permutation group of degree $n$.
Burnside's lemma implies that $G$ contains at least $\#G/n$ elements with a fixed point (the average number of fixed points is $1$ and no element has more than $n$).
The Chebotarev density theorem implies that the density $d(S')$ of the set of rational primes $S'$ for which $K$ has a degree one prime $\p|p$ is at least $1/n$, and all but finitely many of these primes are primes of good reduction for $E$.
As shown in \cite[Cor.~6.2]{Charles}, if $E$ does not have CM the set of supersingular degree one primes of $K$ has density 0 in the set of degree one primes of $K$, which implies $d(S)=d(S')\ge 1/n$.

If $E$ has CM by $\OD$ then it has ordinary reduction at all primes $\p$ of good reduction above a rational prime $p$ for which $\left(\frac{D}{p}\right)=+1$.
The set of such rational primes $p$ has density $1/2$, and if follows that $d(S)\ge d(S')/2=1/(2n)$.
\end{proof}

To simplify notation we shall use ``$\llog$" as an abbreviation for ``$\log\log$".

\begin{proposition}[GRH]\label{prop:Dbound}
Assume GRH, let $D < 0$ be a discriminant and let $h=h(D)$.
If $D$ is fundamental then
\[
|D| \ \le\ \left(\frac{12e^{\gamma}}{\pi}\right)^2 h^2\left(\llog(h+1)+4\right)^2,
\]
and for any $D$ we have
\[
|D| \ \le\ \left(\frac{12e^{2\gamma}}{\pi}\right)^2 h^2\left(\llog(h+1)+4\right)^4,
\]
where $\gamma=0.57721566\ldots$ is Euler's constant.
\end{proposition}
\begin{proof}
Under GRH we may apply \cite[Corollary 1]{LanguascoTrudgian} to obtain
\[
h \ge \frac{\pi}{12e^\gamma}\sqrt{|D|}\left(\llog|D|-\log 2 + \frac{1}{2}+\frac{1}{\llog |D|} + \frac{14\llog |D|}{\log |D|}\right)^{-1},
\]
for all fundamental discriminants $D<-4$.  For $|D|\ge 2^{28}$ we can simplify this to
\begin{equation}\label{eq:hbound}
h \ge \frac{\pi}{12e^\gamma}\sqrt{|D|}\left(\llog|D|+2.3\right)^{-1}.
\end{equation}
We have $\llog |D|+2.3\le |D|^{2/23}$ for $|D|\ge 2^{28}$, in which case $h\ge (\pi e^{-\gamma}/12)|D|^{19/46}$.
Taking logarithms yields $\log h\ge \frac{19}{46}\log |D| + \log (\pi e^{-\gamma}/12) \ge \frac{7}{23}\log |D|$,  and taking logarithms again yields $\llog|D|\le \llog h+1.2$, which implies
$\llog|D|+2.3\le \llog h + 3.5 \le \llog(h+1)+4$.
Thus $h \ge (\pi e^{-\gamma}/12)\sqrt{|D|}(\llog(h+1)+4)^{-1}$, which implies that the proposition holds for all fundamental discriminants with $|D|\ge 2^{28}$.

Using the data from~\cite{JacobsonMosunov} available at~\cite{LMFDB} in combination with \eqref{eq:hD}, one can verify that \eqref{eq:hbound} and the first bound of the proposition hold for fundamental $|D| < 2^{28}$, hence for all fundamental discriminants.

Now suppose $D=f^2D_0$ with $f\ge 2$ and $D_0 < -4$ fundamental, and let $h_0\coloneqq h(D_0)$.
Equation~\eqref{eq:hD} implies
\begin{equation}\label{eq:hratiobound}
\frac{h}{h_0} \ge \prod_{p|f}\left(1-\frac{1}{p}\right )f =\phi(f)\ge \frac{f}{e^\gamma(\llog f+2)},
\end{equation}
where (3.42) in \cite{RosserSchoenfeld} is used to obtain $\phi(n)\ge n/(e^\gamma(\llog(n)+2))$ for $n\ge 2$.
We now apply \eqref{eq:hbound} to $h_0$ and $D_0$, then multiply by \eqref{eq:hratiobound} to obtain
\begin{equation}\label{eq:hhbound}
h \ge \frac{\pi}{12e^{2\gamma}}\sqrt{|D|}(\llog|D|+2.3)^{-2},
\end{equation}
where we have use $\llog|D|+2.3$ as an upper bound on both $\llog|D_0|+2.3$ and $\llog f+2$.  Equation \ref{eq:hhbound} holds for all discriminants $D=f^2D_0$ with $f\ge 2$ and $D_0<-4$, but \eqref{eq:hbound} implies that it also holds for $f=1$, and for $D_0=-3,-4$ the RHS of \eqref{eq:hbound} is less than 1/3, which bounds the factor $1/[\OK^\times:\OD^\times]$ in \eqref{eq:hD} that we omitted from \eqref{eq:hratiobound} for $D_0<-4$, so in fact \eqref{eq:hhbound} holds for all discriminants.

The argument that \eqref{eq:hhbound} implies the second bound in the proposition proceeds as in the argument for \eqref{eq:hbound}, except now have $h\ge (\pi e^{-2\gamma}/12)|D|^{15/46}$ and $\log h \ge \frac{9}{46}\log|D|$ and $\llog|D|\le \llog h+1.7$, which allows us to replace $\log |D|+2.3$ with $\llog(h+1)+4$ in \eqref{eq:hhbound}, yielding the second bound in the proposition.
\end{proof}

\begin{proposition}[GRH]\label{prop:splitprimes}
Assume GRH and let $D$ be an imaginary quadratic discriminant.
Then there is a prime $\ell \le(1.075\log|D|+19)^2$ such that $\left(\frac{D}{\ell}\right)=+1$.
\end{proposition}
\begin{proof}
Apply the main theorem of \cite{GrenieMolteni} with $k=0$ and $|G|/|C|=2$.
\end{proof}

We denote by $\H$ the \defi{logarithmic height} of a nonzero polynomial~$H\in\Z[X]$; that is, the logarithm of the maximum absolute value of the coefficients of $H$.
\begin{corollary}[GRH]\label{cor:toobig}
Assume GRH and let $H\in \Z[X]$ be an HCP of degree~$h$.
Then
\[
\H \le 235\,h\log(h+1)^2(\llog(h+1)+4)^2,
\]
and there exists a prime
\[
\ell\le \ell(h)\coloneqq (2.15\log h + 4\log(\llog(h+1)+4)+24)^2
\]
such that the reduction of $H$ is squarefree in $\Fl[X]$ and its roots in $\overline{\F}_\ell$ are not supersingular.
\end{corollary}
\begin{proof}
The first bound is obtained by combining Proposition~\ref{prop:Dbound} with the bounds on $\H$ given in \cite[Theorem 1.2]{Enge} as follows.  Let us first suppose that $|D| \ge 10^6$, in which case it follows from \cite{Klaise,Watkins} that $h\coloneqq h(D) \ge 76$.  It then follows from \cite[Theorem 1.2]{Enge} that
\[
|H_D| \le 3.012\,h + 3.432\sqrt{|D|}\log\Bigl(\!\sqrt{|D|}\Bigr)^2 \qquad\qquad\ \ (\text{for }|D|\ge 10^6),
\]
and Proposition~\ref{prop:Dbound} implies that
\[
\sqrt{|D|}\log\Bigl(\!\sqrt{|D|}\Bigr)^2 \le 67.5\, h\log(h+1)^2(\llog(h+1)+4)^2\qquad(\text{for }|D|\ge 10^6),
\]
where we have used $2.36\log(h+1)$ as an upper bound on $\log(\sqrt{|D|})\le \log (12e^{2\gamma}/\pi)+\log h+2\log(\llog(h+1)+4)$ for $h\ge 76$, and $67.5 > 2.36^2\cdot 12 e^{2\gamma}/\pi$.
Together these imply that the first bound in the corollary holds for $|D|\ge 10^6$, since $235>3.432\cdot 67.5+3.012$, and a direct computation shows that it also holds for $|D|<10^6$.

For the second bound in the corollary, it follows from Propositions~\ref{prop:Dbound} and ~\ref{prop:splitprimes} that if $H=H_D$ then there is a prime $\ell\le\ell(h)$ such that $\left(\frac{D}{\ell}\right)=+1$.
Proposition ~\ref{prop:squarefree} implies that the reduction of $H$ modulo~$\ell$ is squarefree, and $\ell$ splits in $\OD$, so the roots of $H$ are $j$-invariants of elliptic curves with ordinary reduction at $\ell$.
\end{proof}

\subsubsection*{Application to CM testing}
Assuming GRH, Corollary \ref{cor:toobig} provides either a small prime $\ell=O(\log^2 h)$ to which we can apply Propositions~\ref{cor:CM-ell-count} and~\ref{prop:non-CM-mellj} and determine if~$H$ is an HCP or not, or proves that $H$ is not an HCP because no such $\ell\le \ell(h)$ exists.
The bound on $\H$ allows us to simplify the complexity analysis of our algorithms.

\section{Existing Algorithms}\label{sec:oldalgs}

Before presenting our new algorithms, we describe the algorithms previously implemented in \Magma{} and \Sage{} to solve Problems~1 and~2 and review prior work by Charles \cite{Charles}, Achter \cite{Achter}, and Armitage \cite{Armitage}.
To make it easier to compare algorithms, we describe them all in the setting we use here, in which the input is a monic irreducible polynomial $H\in \Z[X]$ and the output is a boolean \texttt{true} or \texttt{false} giving the answer to Problem 1, and in the latter case the discriminant $D$ for which $H=H_D$ if the algorithm also solves Problem~2 (Achter's algorithm gives the fundamental part of the discriminant $D$).
As above, we use $F\coloneqq \Q[X]/(H(X))$, let $j\in F$ denote a root of $H$, and let $E_j$ denote an elliptic curve over $F$ with $j(E_j)=j$.

In order to easily compare complexity bounds we state them in terms of the degree $h$ and the logarithmic height~$\H$. When $H=H_D$ we have
\[
h(D)=O(|D|^{1/2}\log|D|)\qquad\text{and}\qquad \HD=O(|D|^{1/2}(\log|D|)^2),
\]
which under the GRH can be improved to
\[
h(D)=O(|D|^{1/2}\llog |D|)\qquad\text{and}\qquad \HD=O(|D|^{1/2}\log|D|\llog|D|).
\]
We refer the reader to Section 7 and the Appendix of \cite{Sutherland:HCP} for these and other bounds related to $H_D$.
Both bounds are captured by the less precise notation $h(D)=|D|^{1/2+o(1)}$ and $\HD=|D|^{1/2+o(1)}$, ignoring logarithmic factors, which we will use throughout this section in order to simplify the presentation.

The total size of $H$ (in bits) is bounded by $h\H$, which for $H=H_D$ is bounded by $|D|^{1+o(1)}$, and also by $h^{2+o(1)}$; the latter follows from Siegel's theorem \cite{Siegel}, which is ineffective but can be made effective under GRH, as in Proposition~\ref{prop:Dbound}.
As noted in the introduction, $H_D$ can be computed in $|D|^{1+o(1)}=(h\H)^{1+o(1)}$ time.

\subsection{Magma's algorithm}\label{subsec:magma}

\Magma's function {\tt CMtest} (as in \Magma~version 2.27-5) tests whether a given polynomial~$H\in \Z[X]$ is a Hilbert class polynomial $H_D$,
and returns the discriminant $D$ when it is, thereby solving both Problems 1 and 2.
The method used by \Magma{} is based on the following facts about~$H_D$ when~$h(D)>1$:
\begin{itemize}
\item The number of real roots of~$H_D$ is a power of~$2$ (\Magma{} does not exploit the fact that this power of $2$ must divide $h(D)$, as proved in Proposition~\ref{prop:real-roots}).
\item If~$D\equiv0\pmod4$, then the largest real root of~$H_D$ is
\[
j(\sqrt{D}/2) \ge j(\sqrt{-5}) > 1264538.
\]
\item If~$D\equiv1\pmod4$, then~$H_D$ has no real roots greater than~$1728$, and the smallest real root of $H_D$ is
\[
j((1+\sqrt{D})/2) \le j((1+\sqrt{-15})/2) < -191657.
\]
\item If $\tau$ lies in the usual fundamental region for~$\SL_2(\Z)$ acting on the upper half-plane and~$q=e^{2\pi i\tau}$, then~$|q|$ is small, in which case~$j(\tau)\approx q^{-1}+744$, since for $\Im(\tau)\ge\sqrt{|D|}/2$ with~$|D|\ge15$ we have $|q|\le\exp(-\pi\sqrt{15}) < 5\cdot10^{-6}$.
\end{itemize}
The algorithm is as follows, given a monic polynomial~$H\in\Z[X]$ of degree~$h$:
\begin{enumerate}[\hspace{1em}1.]
\item If $h=1$ return (\true{}, $D$) if~$H$ is one of the 13 $H_D$ of degree~$1$, else return \false{}.
\item Compute a bound on the real roots of~$H$ to determine the precision needed.
\item Compute the set of real roots $\RR^+$ of~$H$.
      If $\#\RR^+\ne 2^n$ for some $n$ return \false{}.
\item If $\max(\RR^+)>1728$ set $s\coloneqq 0$, set $r\coloneqq \max(\RR^+)$, and return \false{} if $r\le 1264538$.
      If $\max(\RR^+)\le 1728$ set $s\coloneqq 1$, set $r\coloneqq -\min(\RR^+)$, and return \false{} if $r\le 191657$.
\item Let $D=-\round((\log((-1)^s(r-744)/\pi))^2)$; return \false{} if $D\not\equiv s\pmod4$.
\item Return (\true{}, $D$) if $h=h(D)$ and $H = H_D$, else return \false{}.
\end{enumerate}

\begin{remark}
Computing the real roots is by far the most expensive step, due to the high precision that is used.
In practice one typically finds \emph{a posteriori} that one could have used much less precision, but proving this \emph{a priori} is challenging.
See \cite{Armitage} for a discussion of a related approach that includes explicit precision bounds.
\end{remark}

\subsection{Sage's algorithms}\label{subsec:sage}

Version 9.7 of \Sage{} provided two algorithms, both written by the first author.  Note that since version 10.0 (released in May 2023) the default algorithm used in \Sage{} is our implementation of Algorithm~\ref{alg:two}.

The first method implemented in version 9.7 is only available for degrees $h\le 100$.
It uses a precomputed list of negative discriminants~$D$ with~$h(D)=h\le 100$, based on Klaise's extension in~\cite{Klaise} of Watkins's determination in~\cite{Watkins} of the imaginary quadratic fields of class number~$h\le100$.
Once all the $H_D$ of a fixed degree~$h$ have been precomputed, we may compare any input $H\in \Z[X]$ of degree~$h$ with every possible $H_D$.
As noted in the introduction, it is not practical to precompute all the $H_D$ of degree $h$ when $h$ is large, but for small $h$ this method works quite well.

The second method implemented in version 9.7 is based on the ideas in \S\ref{subsec:GalRep}.
It takes as input an algebraic integer~$j$, computes its minimal polynomial $H$, and constructs the field~$F=\Q(j)$ and an elliptic curve~$E$ over~$F$ with $j(E)=j$.
It then takes up to twenty degree~$1$ primes~$\p$ of~$F$ of good reduction for~$E$ and norm at most $1000$, and computes the Frobenius trace~$a_{\p}$ of the reduction of~$E$ modulo~$\p$, and the discriminant~$\delta_{\p}=a_{\p}^2-4N(\p)$, skipping~$\p$ if $a_{\p}=0$ or~$\delta_{\p}=0$.
For each remaining~$\p$ the squarefree part of~$\delta_{\p}$ is determined.
A negative answer is returned if any two of these squarefree parts are distinct.
Otherwise the computed~$\delta_{\p}$ all have squarefree part~$d$.  The algorithm sets~$D_0=d$ if $d\equiv1\pmod4$ and $D_0=4d$ otherwise, and considers all candidate discriminants $D=D_0f^2$ such that $D\mid\delta_{\p}$ for all the nonzero $\delta_\p$ computed.
For each such~$D$, the algorithm checks whether $h(D)=\deg(H)$ and if so, computes $H_D$ and compares it to $H$.

\subsection{Charles' Algorithms}\label{subsec:Charles}

In a 2004 preprint~\cite{Charles}, Denis Charles gives three algorithms for testing whether an elliptic curve defined over a number field has CM, which we summarize below.
Our description of Charles' algorithms is adapted to our setting, where the input is a monic irreducible polynomial $H\in \Z[X]$.

\subsubsection{Direct approach}

Given $H\in \Z[X]$ of degree $h$, compute a bound $B$ such that $h(D)>h$ for all $|D|>B$, compute $H_D$ for all $|D|\le B$, and compare each $H_D$ to $H$.
This can be implemented as a deterministic algorithm that runs in time $Bh^{2+o(1)}$ using $B^{1+o(1)}$ space.
As noted in the introduction, one can deterministically compute $H_D$ in time $|D|^{1+o(1)}$, which is bounded by $h^{2+o(1)}$ for $h(D)\le h$, since $h(D)=O(|D|^{1/2}\log|D|)$.
Unconditional values for $B$ are exponential in $h$, but if one assumes GRH then Proposition~\ref{prop:Dbound} yields $B=h^{2+o(1)}$ and a running time of $h^{4+o(1)}$ using $h^{2+o(1)}$ space (the complexity depends only on $h$ because each $H_D$ has size $h^{2+o(1)}$, and we can test whether $H=H_D$ in $h^{2+o(1)}$ time, no matter how big $\H$ is).

\subsubsection{Randomized supersingularity testing}

Pick a random large prime $p$ and a random prime $\p$ of $F$ above $p$, construct $\overline E_j/\F_\p$, and use Schoof's algorithm \cite{Schoof} to compute $a_\p=N(\p)+1-\#\overline E_j(\F_\p)$.
If $a_\p=0$ report that $H$ likely is an HCP and otherwise report that it is not.
As written this is not quite a Monte Carlo algorithm, its success probability is not strictly bounded above 1/2, but if one performs the test twice and reports that $H$ is an HCP if either Frobenius trace is zero, it will be correct with probability close to 3/4 and strictly above 1/2 whenever $H$ is an HCP.
Serre's $O(x/(\log x)^{3/2+\epsilon})$ upper bound on the number of supersingular primes of norm bounded by $x$ for a non-CM elliptic curve over a number field \cite{Serre} implies that the algorithm will erroneously report that $H$ is an HCP with probability tending to $0$ as $x\to \infty$ and below $1/2$ for all sufficiently large $x$, yielding a Monte Carlo algorithm with two-sided error.
Charles chooses $p$ so that $\log p \approx h^2+\H$, which implies an expected running time of $(h^2\H)^{5+o(1)}$ using $(h^2\H)^{3+o(1)}$ space \cite[Cor.~11]{SS}.
Under the GRH, the time complexity can be improved to $(h^2\H)^{4+o(1)}$ time \cite[Cor.~3]{SS} and $(h^2\H)^{2+o(1)}$ space using Elkies improvement to Schoof's algorithm and using the space efficient algorithm of \cite{Sutherland:modpoly} to compute instantiated modular polynomials.

\subsubsection{Galois representations}

Pick a prime $\ell\ge 5$ that is large enough to guarantee that every non-CM elliptic curve defined over a number field of degree $h$ whose $j$-invariant has minimal polynomial with coefficients bounded by $\H$ has Galois representation $\overline \rho_{E,\ell}$ with image containing $\SL_2(\Fl)$.
Construct an elliptic curve $E/F$ whose $j$-invariant is a root of $H$, where $F:=\Q[X]/(H(X))$, and check if $\Gal(F(E[\ell])/F)$ is solvable.
Charles does not attempt to make the bound on $\ell$ explicit, but explicit lower bounds on $\ell$ can be derived from \cite[Cor.\ 1]{Pellarin}.
These bounds grow quadratically with the Faltings height of $E_j$ and include an $h^{4+o(1)}$ factor with a leading constant of $10^{78}$.
The degree of the number field $F(E[\ell])$ will be much larger than this, typically on the order of $\ell^4$.
Even if one assumes the GRH, this approach is unlikely to lead to a practical algorithm.

\subsection{Achter's Algorithm}\label{subsec:Achter}

Achter's algorithm determines whether $E_j$ has complex multiplication, and if so, computes the CM field $K$, but not the endomorphism ring.
The approach sketched by Achter in ~\cite{Achter} involves working over an extension $L/F$ where the base change $E_j'$ of $E_j$ to $L$ has everywhere good reduction, finding a prime~$\p$ of good ordinary for reduction $E_j'$, and computing the endomorphism algebra~$K$ of the reduction of $E_j'$ to the residue field $\F_\p$.
The algorithm then verifies that for every prime $\q$ of $K$ with $N(\q)\le B$, either the reduction of $E_j'$ at $\q$ is supersingular, or its endomorphism algebra is isomorphic to~$K$.
Under the GRH we can take $B=(h\H)^{2+o(1)}$, assuming $\log |\disc(L)|=O(h\H)$.
Provided one efficiently reduces~$H$ modulo many primes~$\q$ in batches using a product tree (as we do in our second algorithm), the running time is $B^{1+o(1)}=(h\H)^{2+o(1)}$ using $(h\H)^{1+o(1)}$ space.
Without the GRH the algorithm takes exponential time.

\subsection{Armitage's Algorithm}\label{subsec:Armitage}

In a 2020 preprint~\cite{Armitage}, Armitage gives a deterministic algorithm for inverting the $j$-function, with CM testing as an application.
This is based on ideas similar to those used in \Magma's algorithm, and in our notation takes $(h^2\H)^{1+o(1)}$ time and space (Armitage gives a more precise bound that makes all the logarithmic and doubly logarithmic factors explicit).

\section{Two new algorithms}\label{sec:newalgs}

We now present two new algorithms for identifying Hilbert class polynomials.
Both take an irreducible monic polynomial $H\in\Z[X]$ as input and return \true\ if $H = H_D$ for some imaginary quadratic discriminant $D$ and \false\ otherwise.
The second algorithm also determines the value of $D$ in the case $H=H_D$.
Both algorithms can be used to determine whether an elliptic curve $E$ over a number field admits complex multiplication (possibly defined over a quadratic extension), by taking $H$ to be the minimal polynomial of the $j$-invariant $j(E)$.
The second algorithm can be used to compute the geometric endomorphism ring of $E$, which will be $\Z$ if the algorithm returns \false, and the unique imaginary quadratic order of discriminant $D$ otherwise.
Once the geometric endomorphism ring $\End(E)$ of $E/F$ known, one can easily determine the ring of endomorphisms defined over $F$: it is $\End(E)$ when $\sqrt{\disc(\End(E))}$ lies in $F$ and $\Z$ otherwise.

We should note that the correctness of our first algorithm, which is deterministic, depends on the generalized Riemann hypothesis (GRH).
Our second algorithm is probabilistic, of Las Vegas type, and the correctness of its output does not depend on the GRH, although we will assume the GRH and some other heuristics when analyzing its complexity.

In the descriptions of our algorithms below, $\ss_\ell\in \Fl[x]$ denotes the supersingular polynomial in characteristic $\ell$, $\gcd(a,b)$ denotes the monic greatest common divisor of two univariate polynomials over a field, and $\Res(a,b)$ denotes the resultant of two univariate polynomials; see \cite[Ch.~12]{GKZ} for several equivalent definitions of $\Res(a,b)$.
For $a,b\in \Z[X,Y]$ we use $\Res_Y(a,b)$ to denote the resultant of $a,b\in \Z[X][Y]$, viewed as univariate polynomials in $Y$ with coefficients in $\Z[X]$.

\begin{algorithm}\label{alg:one}
Given a monic irreducible $H\in \Z[X]$ of degree $h$, return \true\ if $H$ is a Hilbert class polynomial and \false\ otherwise.
\end{algorithm}

\noindent
For increasing odd primes $\ell$:

\begin{enumerate}[\hspace{1em}1.]
\item If $\ell$ exceeds the bound $\ell(h)$ given by Corollary~\ref{cor:toobig}, return \false.\label{alg1:toobig}
\item Compute $\overline H:= H\bmod \ell\in \Fl[X]$, and if $\gcd(\overline H,\overline H')\ne 1$ proceed to the next $\ell$.\label{alg1:bad}
\item Compute $\ss_\ell\in \Fl[x]$, and if $\overline H$ divides $\ss_\ell$ proceed to the next $\ell$.\label{alg1:ss}
\item Compute the classical modular polynomial $\Phi_\ell\in \Z[X,Y]$.\label{alg1:phi}
\item Compute the resultant $G(X)\coloneqq \Res_Y(\Phi_\ell(X,Y),H(Y))\in \Z[X]$.\label{alg1:res}
\item If $G$ is not divisible by $H$ then return \false.\label{alg1:divf}
\item If $G/H\in\Z[X]$ is not divisible by $H$ then return \false.\label{alg1:divf2}
\item Return \true.\label{alg1:done}
\end{enumerate}

\begin{lemma}\label{lem:abD}
Let $D<0$ be a discriminant, and let $\ell$ be the least odd prime such that $H_D\bmod \ell$ is squarefree in $\Fl[x]$ and does not divide $\ss_\ell$.
Then $\left(\frac{D}{\ell}\right)=+1$.
\end{lemma}
\begin{proof}
First suppose that $D$ is not one of the 101 abelian discriminants.
Then any prime unramified in $F\coloneqq \Q[x]/(H_D(X))$ will be unramified in the ring class field~$L$ for the order $\OD$, since $L$ is the Galois closure of $F$ when $D$ is not abelian.
This implies that $\ell\nmid D$, since the ring class field $L$ is an abelian extension of $K\coloneqq \Q(\sqrt{D})$ that is ramified at every odd prime $\ell$ dividing $D$ (for $D'$ a discriminant not divisible by $\ell$ and $D=\ell^2D'$ we must have $h(D)>h(D')$ by equation~\eqref{eq:hD}, and this forces ramification at $\ell$; the cases $D'\ge -4$ are ruled out by $3|D'=-3$ and $h(-36)\ge h(-4)$).
Thus, in the nonabelian case, if $H_D\bmod \ell$ is squarefree then $\ell$ is unramified in $F$ and $\ell\nmid D$.
If $H_D\bmod \ell$ also does not divide $\ss_\ell$, then its roots are $j$-invariants of elliptic curves with ordinary reduction at primes above $\ell$, so $\ell$ is not inert in $K$ and $\left(\frac{D}{\ell}\right)=+1$.

Computing the least prime $\ell$ satisfying the hypothesis for each of the 101 abelian~$D$ shows that $\left(\frac{D}{\ell}\right)=+1$ in all of these cases as well.
\end{proof}

\begin{theorem}[GRH]\label{thm:alg1correct}
If Algorithm~\ref{alg:one} returns \true{} then $H$ is a Hilbert class polynomial.
Under GRH, if Algorithm~\ref{alg:one} returns \false{} then $H$ is not a Hilbert class polynomial.
\end{theorem}
\begin{proof}
Corollary~\ref{cor:toobig} implies that if Algorithm~\ref{alg:one} returns \false{} in step~\ref{alg1:toobig} then $H$ cannot be an HCP unless the GRH is false.
Otherwise, the algorithm reaches step~\ref{alg1:phi} for a prime $\ell\le \ell(h)$ for which $\overline H$ is squarefree and does not divide $\ss_\ell$.
Let $j\in \Qbar$ be a root of $H$ and let~$E_j$ be an elliptic curve with $j$-invariant $j$ that has good reduction at $\ell$ (such an~$E_j$ exists since $j$ is integral).
Then $E_j$ has good ordinary reduction at $\ell$.

If $H$ is not an HCP then Proposition~\ref{prop:non-CM-mellj} implies $m_\ell(j)<2$ and the algorithm returns \false{} in step \ref{alg1:divf} or \ref{alg1:divf2}.
Otherwise, $H=H_D$ for some $D$ with $\left(\frac{D}{\ell}\right)=+1$, by Lemma~\ref{lem:abD}.
Proposition~\ref{cor:CM-ell-count} then implies $m_\ell(j) = 2$, and the algorithm returns \true.
\end{proof}

\begin{remark}
If the input polynomial $H$ is the minimal polynomial of the $j$-invariant of a $\Q$-curve $E$, the algorithm could in principle reach step~\ref{alg1:divf2} (but not step~\ref{alg1:done}).
However, we are not aware of any case where this happens!
In every example we have checked, $H$ is not squarefree modulo the prime degree $\ell$ of any isogeny between Galois conjugates of $E$.
If it were known that this never happens, then step~\ref{alg1:divf2} could be removed (but this would not significantly change the running time).
\end{remark}

\begin{remark}
In our intended application $H$ is the minimal polynomial of an algebraic integer $j(E)$ which is known \emph{a priori} to be irreducible.
To obtain an algorithm that works for any input polynomial one should add an irreducibility test.
The cost of this test will be negligible in almost any practical setting, but in the worst case it could dominate the complexity.
To see the necessity of irreducibility, note that Algorithm~\ref{alg:one} will return \true\ when~$H$ is a product of Hilbert class polynomials whose roots are $j$-invariants of elliptic curves for which the least prime of ordinary reduction not dividing $\disc(H)$ is the same for every factor of $H$.
This occurs when $H=H_{-43}H_{-63}$, for example.
\end{remark}

Recall that $\H$ denotes the logarithmic height of a polynomial $H\in\Z[X]$, the logarithm of the maximum of the absolute values of its coefficients.

\begin{theorem}\label{thm:alg1complexity}
Algorithm~\ref{alg:one} can be implemented as a deterministic algorithm that runs in $(h^2\H)^{1+o(1)}$ time using $(h\H)^{1+o(1)}$ space.
\end{theorem}
\begin{proof}
For $\ell\le \ell(h)=O(\log^2\!h)$ the time to compute the reduction $\overline H$ of $H$ modulo $\ell$ is $(h\H)^{1+o(1)}$.
Once $\overline H$ has been computed, the time to compute $\gcd(\overline H,\overline H')$ and the supersingular polynomial $\ss_\ell$ (via the formula in \cite[Thm.~1]{Finotti}, for example), are both polynomial in $\ell$, hence polylogarithmic in $h$.
The number of iterations is $O(\log^2\!h)=h^{o(1)}$, so the total time needed to reach step~\ref{alg1:phi} is bounded by $(h\H)^{1+o(1)}$, as is the space.
The modular polynomial $\Phi_\ell$ can be computed using the deterministic algorithm described in \cite{BCRS} in $\ell^{4+o(1)}=h^{o(1)}$ time using $h^{o(1)}$ space.

In step~\ref{alg1:res} we may compute $\Res_Y(\phi_\ell(X,Y),H(Y))$ as the determinant of the corresponding Sylvester matrix with entries in $\Z[X]$.
This is a square $(h+\ell+1)\times (h+\ell+1)$ matrix with $2h(\ell+1)$ nonzero entries; it is sparse, since $\ell=h^{o(1)}$.
We can compute the determinant by making the matrix upper triangular using $h^{1+o(1)}$ row operations, each of which involves $h^{1+o(1)}$ ring operations in $\Z[X]$ on polynomials of degree $O(\ell)$.
The total time is bounded by $(h^2\H)^{1+o(1)}$, and if we delete all but the diagonal entry in each row once it is no longer needed, the space is $(hH)^{1+o(1)}$.
\end{proof}

\begin{remark}
The complexity of Algorithm 1 (both in theory and in practice) is dominated by the resultant computation in step 5.
In the proof of Theorem~\ref{thm:alg1complexity} use naive row reduction to argue that this can be done in $(h^2\H)^{1+o(1)}$ time.
Applying the faster algorithm of van der Hoeven and Lecerf \cite{vanderHoevenLecerf} would improve this to complexity bound to $(h\H)^{1+o(1)}$ if one could prove that the polynomials $\Phi_\ell(X,Y)$ and $H(Y)$ are ``sufficiently generic''.  We have not attempted to do this.
\end{remark}

\begin{remark}\label{rem:auxp}
The practical performance of Algorithm 1 on inputs that are not HCPs can be improved without significantly impacting its performance on inputs that are HCPs by computing the resultant $G$ over a finite field $\Fp$ for some small auxiliary prime $p$ and checking the divisibility conditions in steps~\ref{alg1:divf} and~\ref{alg1:divf2} for $H$ in $\Fp$ before computing the resultant over $\Z$.
If one uses a CRT-based resultant computation this can be done during the resultant computation at no additional cost.
\end{remark}

We now present our second algorithm, which is not deterministic (an efficient implementation will use Rabin's algorithm \cite{Rabin:roots} to find roots of polynomials in $\Fp[x]$), but has several advantages over Algorithm~\ref{alg:one}: it is significantly faster in practice, it solves both Problems 1 and 2, and its correctness does not depend on the GRH.
To simplify the presentation we elide some details but include references to relevant remarks.

\begin{algorithm}\label{alg:two}
Given a monic irreducible $H\in \Z[X]$ of degree $h$, determine if $H=H_D$ for some $D$.
If so, return \true{} and the value of $D$, otherwise return \false{}.
\end{algorithm}

\noindent
Let $\mathcal D$ be the set of integers $h^+|h$ that are powers of $2$ of the same parity as $h$.\\
Let $\pmin\coloneqq \lceil 37 h^2(\llog(h+1)+4)^4\rceil$ (see Remark~\ref{rem:pmin}).\\
For increasing primes $p\ge p_{\rm min}$ (see Remark \ref{rem:batching}) :
\begin{enumerate}[\hspace{1em}1.]
\item Compute $\overline H\coloneqq H\bmod p\in \Fp[x]$.
\item Compute $d\coloneqq \deg\gcd(\overline H(x),x^p-x)$.
\item If $d=0$ then proceed to the next prime $p$.
\item If $\gcd(\overline H,\overline H')\ne 1$ then proceed to the next prime $p$.\label{alg2:gcd}
\item If $d < h$ and $d\not\in \mathcal D$ then return \false.\label{alg2:checkd1fac}
\item Let $\overline E/\Fp$ be an elliptic curve whose $j$-invariant is a root of $\overline H$.\label{alg2:getE}
\item If $\overline E$ is supersingular then proceed to the next prime $p$.\label{alg2:sstest}
\item Compute $D\coloneqq \disc(\End(\overline E))\in \Z$ (see Remark~\ref{rem:end}).\label{alg2:endE}
\item If $h(\disc\End(\overline E))\ne h$ return \false{}, otherwise compute $H_D(X)$ (see Remark~\ref{rem:noHD}).\label{alg2:hD}
\item If $H=H_D$ then return \true{} and $D$; otherwise return \false.\label{alg2:final}
\end{enumerate}

As shown in the proof of Theorem~\ref{thm:alg2correct} below, Algorithm~\ref{alg:two} will terminate once it finds a prime $p\ge \pmin$ for which the reduction of $H$ modulo $p$ is squarefree and has an $\Fp$-rational root that is the $j$-invariant of an ordinary elliptic curve. The existence of such a prime is guaranteed by Proposition~\ref{prop:ordprimes}.

\begin{theorem}\label{thm:alg2correct}
Algorithm~\ref{alg:two} terminates on all inputs, and it returns \true{} and $D$ if and only if $H$ is the Hilbert class polynomial $H_D$.
\end{theorem}
\begin{proof}
Algorithm~\ref{alg:two} only returns \true{} in step~\ref{alg2:final} after it has verified that $H=H_D$, and it is clearly correct in this case.
It remains only to show that the algorithm always terminates, and that it never returns \false{} when $H$ is an HCP.
Proposition~\ref{prop:ordprimes} implies that the set of degree one primes of $F\coloneqq \Q[X]/(H(X))$ of good ordinary reduction for any elliptic curve $E_j$ whose $j$-invariant is a root of $H$ has positive density, and the algorithm will terminate once it reaches such a prime.

Now suppose $H$ is an HCP $H_D$.
Propositions \ref{prop:real-roots} and~\ref{prop:mod-p-factors} imply that the algorithm cannot return \false\ in step \ref{alg2:checkd1fac}.
Proposition~\ref{prop:mod-p-factors} implies that when Algorithm~\ref{alg:two} reaches step~\ref{alg2:endE} the prime $p$ splits completely in $F\coloneqq \Q[X]/(H(X))$, and in $K=\Q(\sqrt{D})$, since $\overline E$ is ordinary.
It follows that $p$ splits completely in the ring class field $L$ (both in the abelian and nonabelian cases), so the roots of $\overline H$ are $j$-invariants of elliptic curves over $\Fp$ with CM by $\OD$.
The algorithm will thus compute $\disc(\End(\overline E))=D$, find $h(D)=h$, compute $H_D=H$, and correctly return \true{} and $D$.
\end{proof}

\begin{remark}\label{rem:batching}
To achieve the desired asymptotic complexity, rather than iteratively enumerating primes $p$ and computing each $\overline H\coloneqq H\bmod p$ individually, one should use a product/remainder tree to compute reductions $\overline H$ modulo a set of primes that is large enough for their product to have logarithmic height greater than $\H$.
This is not a purely theoretical asymptotic issue, it will yield noticeably better performance when $\H$ is large but still well within the practical range.

In order to (heuristically) apply the Chebotarev density theorem in our complexity analysis, we also want to use random primes, say sets of size $10\max(h,\H)$) chosen from increasing intervals (of length $\approx \pmin$, say) above $\pmin$.
To minimize the time spent on primality proving (which takes $(\log p)^{4+o(1)}$ expected time via~\cite{Bernstein}) one should use Miller-Rabin primality tests \cite{Miller,Rabin:primality} that cost only $(\log p)^{2+o(1)}$ to generate random integers in the desired interval that are very likely to be prime.
This procedure may be omitted in practical implementations by those willing to assume the stronger heuristic that sequential primes behave like random ones.
\end{remark}

\begin{remark}\label{rem:pmin}
The value of $p_{\min}$ used in the algorithm ensures that $4p>|D|$ when $H=H_D$ under GRH (via Proposition~\ref{prop:Dbound}) and it serves two purposes, one practical and one theoretical.
The practical purpose is to prevent the algorithm from wasting time on primes that are too small (a Hilbert class polynomial $H_D$ cannot split completely in $\Fp[x]$ when $4p<|D|$).
The theoretical purpose is to ensure that when $H$ is an HCP we can use a GRH-based explicit Chebotarev bound to argue that after testing $\approx 2h(D)$ random primes $p\in [\pmin,2\pmin]$ we can expect to find one that splits completely in the ring class field, as discussed below.
Our heuristic assumptions apply to any $\pmin=h^{2+o(1)}$, so the exact choice of $\pmin$ is not critical, but a good choice may improve performance.
Note that the algorithm will produce the correct output no matter what $\pmin$ is (including when $\pmin=0$).
\end{remark}

\begin{remark}\label{rem:end}
When computing $\End(\overline E)$ in step~\ref{alg2:endE}, we know that if $E$ has CM by $\OD$, then we must have $h(D)=h$.
When applying the GRH-based subexponential time algorithm of \cite{Bisson, BissonSutherland}, the first step is to compute $a_p\coloneqq p+1-\#\overline E(\Fp)$ by counting points and to then factor $a_p^2-4p$ as $v^2D_0$ for some fundamental discriminant~$D_0$.
The discriminant $D=\disc(\End(\overline E))$ is then of the form $D=f^2D_0$ for some divisor $f$ of $v$ which satisfies $h(f^2D_0)=h$.
Our choice of $\pmin$ ensures that for $p\in [\pmin,2\pmin]$ we have $D\ge h^2/\llog^4 h$.
This forces $v/f$ to be quite small: $O(\llog^4 h)$ under GRH.
In other words, among the (possibly numerous) divisors of $v$, there are only a small number of possible $f$, each determined by a small cofactor $c=v/f$.
We can efficiently compute $h(D_0)$ and use \eqref{eq:hD} to compute $h(f^2D_0)$ for all $f|v$ with small cofactor.
If no candidate $f|v$ satisfies $h(f^2D_0)=h$ we can immediately return \false{}, and if only one does, we can assume $D=f^2D_0$, skip the computation of $\End(E)$, and proceed to the next step (which will compute $H_D$ and verify our assumption).
Even when there is more than one candidate $f$, having a short list of candidates for which the ratio $v/f$ is small speeds up the computation of $\End(E)$, reducing it to descents on $\ell$-volcanoes for the small primes $\ell$ dividing $v/f$.  The optimizations described in this remark are incorporated in the implementation of our algorithm described in \S \ref{sec:timings} and provide a noticeable practical speedup but have no impact on the asymptotic complexity (under GRH), since the computation of $\End(\overline E)$ is not the asymptotically dominant step.
\end{remark}

\begin{remark}\label{rem:noHD}
If one knows \emph{a priori} that $H$ is an HCP (via Algorithm~\ref{alg:one}, for example), one can simply return $D$ after step \ref{alg2:endE}, there is no need to compute $H_D$.
\end{remark}

Let us now consider the complexity of Algorithm~\ref{alg:two}.
There is an effective form of the Chebotarev density theorem under GRH due to Lagarias--Odlyzko \cite{LagariasOdlyzko} and Serre \cite{Serre}, that has been made completely explicit by Winckler \cite{Winckler}.
In the case of interest to us the theorem states that for any Galois extension $M/\Q$ with Galois group $G$, any set $C$ that is a union of conjugacy classes of $G$, and $x\ge 2$ we have
\begin{equation}\label{eq:chebotarev}
\left|\pi_C(x)-\frac{|C|}{|G|}\Li(x)\right| \le \frac{|C|}{|G|}\sqrt{x}\left(\left(32+\frac{181}{\log x}\right)\log d_M + \left(28\log x+330 + \frac{1655}{\log x}\right)n_M\right),
\end{equation}
where $\pi_C(X)$ counts primes $p\le x$ with $\Frob_p\subseteq C$, $d_M\coloneqq \disc \OO_M$, and $n_M\coloneqq [M:\Q]$.

We would like to apply \eqref{eq:chebotarev} to the splitting field of $H$ (the ring class field $L$ when $H=H_D$ is nonabelian), with $C$ the union of all conjugacy classes of $G$ that fix at least one root.
The proof of Proposition~\ref{prop:ordprimes} implies $|C|/|G|\ge 1/h$, and \eqref{eq:chebotarev} implies that for all sufficiently large $x$ we should expect to find an $\Fp$-rational root of a squarefree $\overline H\in \Fp[X]$ after trying $h$ primes $p$, on average, and after $2h$ primes we should find a root that is an ordinary $j$-invariant.
Once this occurs the algorithm will compute $D=\disc(\End(\overline E))$ and $H_D$ and terminate.

When $H$ is an HCP we will have $n_M=h$ and $d_M\approx|D|^h$, in which case ``sufficiently large'' means $x\ge ch^2\log^2 h$ for a suitable choice of $c$, which is not much larger than the value of $\pmin$ used by the algorithm (we could increase $\pmin$ to $ch^2\log^2 h$ without changing the complexity of the algorithm).

But if $H$ is not an HCP we could easily have $n_M=h!$ nearly as large as~$h^h$, which makes it infeasible to use primes $p$ that are large enough for us to apply~\eqref{eq:chebotarev}.
However, in practice one finds that $\pmin$ is already sufficiently large: the algorithm reliably finds suitable primes~$p$ after roughly the number of iterations predicted by the Chebotarev density theorem, even when $H$ is not an HCP.
Thus, in order to understand and reliably predict the likely performance of our algorithm in practice, we analyze its complexity under the following heuristic assumption:
\medskip

\begin{enumerate}[(\textbf{H})]
\item For $\pmin\coloneqq\pmin(h)=h^{2+o(1)}$, among the primes $p\in [\pmin,2\pmin]$, the number of~$p$ for which $\overline H$ has a root in $\Fp$ that is the $j$-invariant of an ordinary elliptic curve is bounded below by $c\pmin/(h\log h)$ for some absolute constant $c$ as $h\to \infty$.
\end{enumerate}

We now analyze the asymptotic complexity of Algorithm~\ref{alg:two} under this heuristic assumption, assuming an implementation that uses Remark~\ref{rem:batching} to accelerate the enumeration of primes.

\begin{theorem}[Heuristic]\label{thm:alg2complexity}
Assume the GRH, heuristic (\textbf{H}), and the heuristics listed in \cite[\S7.1]{Sutherland:HCP}.
Algorithm~\ref{alg:two} can be implemented as a Las Vegas algorithm that runs in
\[
h^2(\log h)^{3+o(1)}+h(h+\H)\log(h+\H)^{2+o(1)} = h(h+\H)^{1+o(1)}
\]
expected time, using at most $h(h+\H)\log(h+\H)^{1+o(1)}$ space.
\end{theorem}
\begin{proof}
Under our heuristic assumptions we expect the algorithm to use $O(h)$ primes $p\in [\pmin,2\pmin]$ with $\pmin=h^{2+o(1)}$.
As noted in Remark~\ref{rem:batching}, the expected time to generate $O(h)$ random primes in this interval and rigorously prove their primality is bounded by $h(\log h)^{4+o(1)}$, and if we use product trees and batches of $\H/(\log h)$ primes at a time, we will spend a total of $h(h+\H)(\log (h+\H))^{2+o(1)}$ time computing reductions $\overline H$.
We can process each batch using $O(h\H+\H\log\H)$ space by using a separate product tree of size $O(H\log H)$ space for each coefficient and noting that the total size of $\H/\log p$ polynomials $\overline H$ is $O(h\H)$ bits.
The size of the HCP $H_D$ computed in step~\ref{alg2:final} is $O(h^2(\log h)^{1+o(1)})$ bits (under GRH), and this is also bounds the space to compute it via~\cite{Sutherland:HCP}.
All other steps take negligible space, and this completes the proof of the space bound.

The time to compute $\overline H'$ and $\gcd(\overline H,\overline H')$ is $h(\log h)^{3+o(1)}$ using the fast Euclidean algorithm, and the time to compute $\gcd(\overline H(x),x^p-x)$ is bounded by $h(\log h)^{3+o(1)}$ if we work in the ring $R\coloneqq \Fp[x]/(\overline H(x))$ and use binary exponentiation to compute $x^p\bmod \overline H$ and use Kronecker substitution with fast multiplication to multiply in~$R$.
The total time spent in steps \ref{alg2:gcd}--\ref{alg2:checkd1fac} is then bounded by $h^2(\log h)^{3+o(1)}$.

We expect to perform steps \ref{alg2:getE}--\ref{alg2:final} once or twice on average, since the density of ordinary primes among the degree 1 primes of $F=\Q[X]/(H(X))$ is at least 1/2 by Proposition~\ref{prop:ordprimes} and its proof.
The expected time to find a single root of $\overline H$ in step \ref{alg2:getE} using \cite{Rabin:roots} is bounded by $h(\log h)^{3+o(1)}$.
The time to construct $E$ is negligible, as is the $(\log p)^{3+o(1)}$ time for supersingularity testing via \cite{Sutherland:SS} in step \ref{alg2:sstest}.
The time to compute $\End(E)$ in step~\ref{alg2:endE} via \cite{Bisson,BissonSutherland}, including the time to factor $a_p^2-4p$, is subexponential in $\log h$, which is also negligible.
The time to compute $h(D)$ in step~\ref{alg2:hD} is also subexponential in $\log h$, via the algorithm in \cite{HafnerMcCurley}.
The computation of $H_D$ in step~\ref{alg2:final} using \cite{Sutherland:HCP} takes $h^2(\log h)^{3+o(1)}$ time under the heuristics in \cite[\S7.1]{Sutherland:HCP} and Proposition~\ref{prop:Dbound}.
\end{proof}

\begin{remark}
Theorem~\ref{thm:alg2correct} applies to the implementation of Algorithm~\ref{alg:two} described in Theorem~\ref{thm:alg2complexity}.
Its output is correct whether the heuristic assumptions in the hypothesis of Theorem~\ref{thm:alg2complexity} hold or not.
\end{remark}

\begin{remark}
Our algorithms can be adapted to identify polynomials~$H$ associated to \emph{class invariants} other than the $j$-function, including many of those commonly used.
We refer the reader to \cite{BLS} and \cite{EngeSutherland} for a discussion of various alternative class invariants $g$ and the methods that can be used to compute corresponding class polynomials $H^g\in \Z[X]$ and modular polynomials $\Phi^g\in \Z[X,Y]$.
The modular function $g$ will satisfy some relation with $j$ encoded by a polynomial $\Psi^g\in \Z[X,j]$.
For example, for $\gamma_2\coloneqq \sqrt[3]{j}$ we have $\Psi^{\gamma_2}(X,j)=X^3-j$, and for the Weber function $\f(z)\coloneqq \zeta_{48}^{-1}\eta((z+1)/2)/\eta(z)$ we have $\Psi^{\mathfrak f}(X,j)=(X^{24}-16)^3-X^{24}j$.
To apply our algorithms to $g$ rather than $j$ one uses the polynomial $\Psi^g$ to replace $j$-invariants with $g$-invariants, replaces $H_D$ with $H_D^g$ (for suitable $D$), replaces $\Phi_\ell$ with $\Phi_\ell^g$, and replaces the supersingular polynomial $\ss_\ell$ with a polynomial $\ss_\ell^g$ whose roots are $g$-invariants of supersingular elliptic curves in characteristic~$\ell$.
The main difference is that, depending on the class invariant~$g$, one may need to restrict the discriminants $D$ and the primes $p$ that are used.
See \cite{EngeSutherland}, \cite[\S 7]{BLS}, and \cite[\S 3]{Sutherland:modpoly} for more information.
\end{remark}

\section{Computational results}\label{sec:timings}

Implementations of Algorithms~\ref{alg:one} and~\ref{alg:two} in \Magma{}, \GP{}, and \Sage{} are available in the \Github{} repository associated to this paper \cite{repo}.  In all three cases Algorithm~\ref{alg:two} outperforms Algorithm~\ref{alg:one}, while solving both Problems 1 and 2, so here we report only on our implementation of Algorithm~\ref{alg:two}, as this is the algorithm we would recommend using in practice.
We did not incorporate Remark~\ref{rem:batching}, which makes little practical difference in the range of inputs we tested.  We did incorporate Remarks~\ref{rem:pmin} and~\ref{rem:end}, both of which noticeably improve performance.

In our implementations we rely on built-in functions provided by \Magma{}, \GP{}, and \Sage{} for supersingularity testing and computing Hilbert class polynomials, and to compute the modular polynomials used in our implementation of Remark~\ref{rem:end} (we note that \Sage{} uses \Arb{} \cite{Arb} to compute Hilbert class polynomials and \GP{} to compute modular polynomials).  The time spent on supersingularity testing and computing modular polynomials is negligible, both in practice and asymptotically.

By contrast, the time spent computing $H_D$ in step~\ref{alg2:final} of Algorithm~\ref{alg:two} typically dominates the computation on inputs that are Hilbert class polynomials.
The variation in the running times of our three implementations is mostly explained by differences in how quickly \Magma{}, \GP{}, and \Sage{} compute HCPs, as shown in Table~\ref{tab:timings}.
In the \Magma{} and \GP{} implementations, the total running time is never more than twice the time to compute $H_D$; in other words, Algorithm~\ref{alg:two} can identify HCPs in less time than it takes to compute them.
This is not true of the \Sage{} implementation for the range of discriminants we tested, due to the fact that \Sage{} (via \Arb{}) computes HCPs more quickly than \Magma{} or \GP{} for discriminants in this range,\footnote{For larger $|D|$ the CRT-based algorithm \cite{Sutherland:HCP} used by \GP{} outperforms the complex analytic method \cite{Enge} used by \Arb{}. The crossover appears to occur around $h(D)=5000$ but comes earlier if one exploits all the optimizations described in \cite{Sutherland:HCP}.} but finite field arithmetic in \Sage{} is substantially slower than it is in \Magma{} or \GP{}.

To assess the practical performance of our algorithms we applied them to inputs $H=H_D$ (a polynomial that is an HCP) and $H=H_D+1$ (a polynomial that is not) for various imaginary quadratic discriminants~$D$ with class numbers $h(D)$ up to $1000$; results for $h(D)\le 1000$ can be found in Table~\ref{tab:timings}.
For $h(D)\le 200$ we performed the same computations using the methods implemented in \Magma{} Version 2.25-7 and \Sage{} Version 9.7.
For the \Magma{} computations we used the \texttt{CMTest} function located in the package file \texttt{/Geometry/CrvEll/cm.m} (this function is not part of the public interface but can be accessed by importing it).  This function is quite memory intensive, and we were unable to apply it to discriminants with $h(D) > 200$ using the 128GB of memory available to us on our test platform.

For the \Sage{} computations we used the first method described in \S\ref{subsec:sage}, which compares the input $H$ to all $H_D$ with $h(D)=\deg H$.  To get more meaningful timings we implemented a simplified version that does not cache previously computed results and managed our own dictionary of precomputed discriminants $D$, using the dataset computed by Klaise \cite{Klaise}, which we extended to $h(D)\le 1000$ by assuming the GRH. Under this assumption, there are 6,450,424 discriminants~$D<0$ with $h(D)\le 1000$; these can be found in the file \texttt{cmdiscs1000.txt} in \cite{repo}.  The largest is $|D|=227932027$ with $h(D)=996$.

Proposition~\ref{prop:Dbound} implies that every fundamental $D_0<0$ with $h(D_0)\le 1000$ satisfies $|D_0|\le 1629085093$.  Class numbers for all such $D_0$ can be found in the dataset computed by Jacobson and Mosunov \cite{JacobsonMosunov}, which covers $|D_0|<2^{40}$ and can be downloaded from \cite{LMFDB}. There are 4,115,897 fundamental discriminants $|D_0|\le 1629085093$ with $h(D_0)\le 1000$, and one can use \eqref{eq:hD} to compute the 6,450,424 discriminants~$D<0$ with $h(D)\le 1000$; one could use the same approach for $h(D)\le 24416$, but we did not attempt this.
Variations in the timings in the \Sage{} column of Table~\ref{tab:timings} are explained by variations in the sets $S_h\coloneqq \{D<0:h(D)=h\}$; for example, $\#S_{400}=22366$, $\#S_{425}=956$, $\#S_{450}=6745$, $\#S_{475}=1065$, and $\#S_{500}=10458$.

In our tests we used the discriminants $D$ with $h(D)=5$, 10,\ldots, 45, 50, 75, 100, 125, \ldots, 975, 1000 for which $|D|$ takes the median value among all $D<0$ with the same class number.
Performance comparisons can be found in Table~\ref{tab:timings}, which also lists the degree~$h$ and the logarithmic height $\H$ of the inputs (which are effectively the same for both our CM input $H=H_D$ and our non-CM input $H=H_D+1$).
As noted in the introduction, all three implementations of Algorithm~\ref{alg:two} have a substantial advantage over the implementations previously used by \Magma{} and \Sage{} across the entire range of discriminants on both CM and non-CM inputs, especially the latter.  As anticipated by our complexity analysis, this advantage grows as~$h$ increases.

\begin{table}[ht!]
\begin{center}
\begin{small}
\setlength{\tabcolsep}{2.5pt}
\renewcommand{\arraystretch}{0.99}
\begin{tabular}{rrr|rrr|rrr|rrr|rr|rr}
\multicolumn{3}{p{2.4cm}}{\centering}&\multicolumn{3}{p{2.4cm}}{\centering \Magma{}\\ Algorithm 2}&\multicolumn{3}{p{2.4cm}}{\centering \GP\\ Algorithm 2}&\multicolumn{3}{p{2.4cm}}{\centering \Sage\\ Algorithm 2}&\multicolumn{2}{p{1.8cm}}{\centering\Sage{}\\9.7}&\multicolumn{2}{p{1.5cm}}{\centering\Magma{}\\2.25-7}\\
$h$ & $|H|$ & $|D|$ & $t_{\rm HCP}$ & $t_{\rm CM}$ & \hspace{-20pt}$t_{\rm noCM}$ & $t_{\rm HCP}$ & $t_{\rm CM}$ & \hspace{-20pt}$t_{\rm noCM}$& $t_{\rm HCP}$ & $t_{\rm CM}$ & $t_{\rm noCM}$ & $t_{\rm CM}$ & $t_{\rm noCM}$ & $t_{\rm CM}$ & $t_{\rm noCM}$\\[1pt]\toprule
   5 &   120 &      571 &   0.00 &   0.00 & 0.00 &  0.00 &  0.00 & 0.00 &  0.00 &  0.01 & 0.00 & 0.00 & 0.00 & 0.00 & 0.00\\
  10 &   294 &     2299 &   0.00 &   0.01 & 0.00 &  0.02 &  0.01 & 0.00 &  0.00 &  0.02 & 0.00 & 0.00 & 0.01 & 0.00 & 0.01\\
  15 &   527 &     6571 &   0.00 &   0.00 & 0.00 &  0.01 &  0.01 & 0.00 &  0.00 &  0.01 & 0.00 & 0.01 & 0.01 & 0.01 & 0.01\\
  20 &   843 &     9124 &   0.00 &   0.01 & 0.00 &  0.02 &  0.02 & 0.00 &  0.00 &  0.02 & 0.00 & 0.06 & 0.14 & 0.06 & 0.01\\
  25 &   941 &    25747 &   0.00 &   0.01 & 0.00 &  0.06 &  0.06 & 0.00 &  0.00 &  0.06 & 0.00 & 0.02 & 0.06 & 0.02 & 0.06\\
  30 &  1198 &    21592 &   0.00 &   0.02 & 0.00 &  0.05 &  0.06 & 0.00 &  0.00 &  0.07 & 0.00 & 0.15 & 0.39 & 0.12 & 0.12\\
  35 &  1484 &    42499 &   0.00 &   0.02 & 0.00 &  0.04 &  0.07 & 0.00 &  0.00 &  0.11 & 0.00 & 0.06 & 0.15 & 0.36 & 0.36\\
  40 &  1739 &    34180 &   0.01 &   0.02 & 0.00 &  0.05 &  0.06 & 0.00 &  0.00 &  0.02 & 0.00 & 1.04 & 2.67 & 0.44 & 0.45\\
  45 &  1027 &    60748 &   0.01 &   0.03 & 0.00 &  0.09 &  0.10 & 0.00 &  0.00 &  0.03 & 0.01 & 0.24 & 0.63 & 0.98 & 0.97\\
  50 &  2161 &    64203 &   0.02 &   0.02 & 0.00 &  0.09 &  0.09 & 0.00 &  0.00 &  0.04 & 0.00 & 0.63 & 1.64 & 1.41 & 1.40\\
\midrule
  75 &  3221 &   157051 &   0.07 &   0.10 & 0.00 &  0.14 &  0.17 & 0.00 &  0.01 &  0.14 & 0.00 &    2 &    5 &   11 &   11\\
 100 &  4197 &   249451 &   0.15 &   0.23 & 0.00 &  0.29 &  0.37 & 0.00 &  0.03 &  0.30 & 0.00 &   26 &   64 &   50 &   42\\
 125 &  6202 &   516931 &   0.38 &   0.40 & 0.00 &  0.47 &  0.51 & 0.00 &  0.06 &  0.16 & 0.03 &    7 &   17 &  169 &  169\\
 150 &  6318 &   565852 &   0.48 &   0.58 & 0.00 &  0.45 &  0.56 & 0.00 &  0.07 &  0.34 & 0.00 &   75 &  185 &  317 &  317\\
 175 &  8500 &  1016731 &   0.92 &   1.36 & 0.01 &  0.98 &  1.36 & 0.00 &  0.14 &  1.30 & 0.00 &   29 &   71 &  917 &  914\\
 200 &  9520 &   910539 &   1.32 &   1.86 & 0.00 &  0.77 &  1.24 & 0.00 &  0.19 &  1.21 & 0.00 &  600 & 1480 & 1267 & 1260\\
\midrule
 225 & 11427 &  1480588 &   2.12 &   2.16 & 0.01 &  2.46 &  2.56 & 0.00 &  0.29 &  0.38 & 0.10 &   116 &   288\\
 250 & 12923 &  1757251 &   3.05 &   3.88 & 0.00 &  1.82 &  2.55 & 0.00 &  0.39 &  1.90 & 0.00 &   281 &   700\\
 275 & 13299 &  2366443 &   3.63 &   3.82 & 0.00 &  3.50 &  3.70 & 0.00 &  0.47 &  0.94 & 0.02 &   145 &   365\\
 300 & 14621 &  2127259 &   4.64 &   6.20 & 0.01 &  2.06 &  3.23 & 0.00 &  0.60 &  3.28 & 0.02 &  3087 &  7502\\
 325 & 16683 &  3150331 &   6.47 &  13.86 & 0.00 &  2.91 &  7.89 & 0.00 &  0.79 & 13.40 & 0.00 &   257 &   638\\
 350 & 19546 &  3469651 &   9.78 &  10.22 & 0.01 &  4.13 &  4.34 & 0.00 &  1.12 &  1.72 & 0.01 &  1173 &  2869\\
 375 & 20049 &  4428859 &  10.97 &  11.69 & 0.01 &  5.10 &  5.67 & 0.01 &  1.25 &  2.69 & 0.01 &   747 &  1854\\
 400 & 21707 &  3460787 &  12.90 &  16.99 & 0.00 &  5.91 &  8.45 & 0.00 &  1.50 &  5.66 & 0.00 & 13346 & 31952\\
 425 & 22540 &  6268987 &  15.65 &  21.59 & 0.00 &  8.22 & 11.27 & 0.00 &  1.70 &  9.32 & 0.01 &   708 &  1733\\
 450 & 25814 &  5226388 &  18.67 &  20.52 & 0.00 &  9.67 & 10.79 & 0.00 &  2.25 &  4.58 & 0.00 &  5677 & 13738\\
 475 & 26144 &  7776619 &  21.95 &  23.21 & 0.01 &  9.34 & 10.46 & 0.01 &  2.45 &  4.01 & 0.01 &  1082 &  2640\\
 500 & 28965 &  6423467 &  26.22 &  31.21 & 0.01 &  9.99 & 12.35 & 0.00 &  3.03 &  8.52 & 0.00 & 12032 & 28873\\
\cmidrule{1-14}
 525 & 28822 &  7874131 &  29.05 &  30.33 & 0.02 &  8.54 &  8.66 & 0.01 &  3.21 &  4.73 & 0.23\\
 550 & 29376 &  8427715 &  29.54 &  31.06 & 0.00 & 11.27 & 12.44 & 0.00 &  3.44 &  5.55 & 0.01\\
 575 & 31885 & 10340347 &  38.94 &  44.98 & 0.02 & 14.86 & 17.72 & 0.01 &  4.11 & 14.45 & 0.37\\
 600 & 33802 &  7885067 &  45.68 &  49.61 & 0.01 & 14.97 & 17.57 & 0.01 &  4.73 & 10.93 & 0.02\\
 625 & 35331 & 12907387 &  53.27 &  59.03 & 0.01 & 19.46 & 22.95 & 0.01 &  5.30 & 11.80 & 0.51\\
 650 & 36681 & 12266403 &  55.77 &  59.38 & 0.02 & 25.03 & 26.39 & 0.01 &  5.87 &  9.57 & 0.54\\
 675 & 38661 & 13845211 &  69.36 &  74.36 & 0.01 & 13.85 & 16.52 & 0.01 &  6.65 & 11.97 & 0.01\\
 700 & 39857 & 12955579 &  72.36 &  76.45 & 0.01 & 14.50 & 17.28 & 0.01 &  7.22 & 10.72 & 0.01\\
 725 & 42183 & 17982403 &  89.32 &  97.08 & 0.01 & 25.03 & 30.19 & 0.01 &  8.26 & 17.34 & 0.63\\
 750 & 41262 & 14687500 &  89.52 &\hspace{-4pt} 129.96 & 0.02 & 31.56 & 55.78 & 0.01 &  8.24 & 55.11 & 0.23\\
 775 & 43681 & 19485628 &  99.91 &\hspace{-4pt} 124.65 & 0.01 & 18.52 & 31.65 & 0.01 &  9.29 & 42.91 & 0.01\\
 800 & 44169 & 13330819 & 106.77 &\hspace{-4pt} 122.06 & 0.02 & 20.26 & 28.64 & 0.01 &  9.73 & 27.43 & 0.02\\
 825 & 48830 & 19844179 & 136.32 &\hspace{-4pt} 150.59 & 0.01 & 19.61 & 28.84 & 0.01 & 12.00 & 32.50 & 0.01\\
 850 & 50884 & 21333372 & 138.99 &\hspace{-4pt} 159.28 & 0.01 & 42.15 & 55.92 & 0.01 & 13.29 & 38.59 & 0.01\\
 875 & 53248 & 25154971 & 171.62 &\hspace{-4pt} 192.83 & 0.01 & 30.35 & 42.59 & 0.01 & 14.71 & 43.61 & 0.01\\
 900 & 47449 & 19028875 & 141.95 &\hspace{-4pt} 145.31 & 0.01 & 28.00 & 30.76 & 0.01 & 12.59 & 16.73 & 0.01\\
 925 & 54158 & 27862699 & 189.08 &\hspace{-4pt} 190.54 & 0.03 & 36.09 & 36.15 & 0.02 & 15.95 & 17.90 & 1.01\\
 950 & 54775 & 26553843 & 200.47 &\hspace{-4pt} 206.92 & 0.03 & 37.02 & 40.90 & 0.01 & 16.71 & 24.61 & 0.48\\
 975 & 57027 & 27589459 & 222.11 &\hspace{-4pt} 240.95 & 0.04 & 34.99 & 46.40 & 0.02 & 18.47 & 47.52 & 1.00\\
 1000& 56827 & 23519868 & 215.96 &\hspace{-4pt} 267.94 & 0.03 & 49.48 & 83.42 & 0.02 & 18.81 & 81.98 & 0.03\\
\cmidrule{1-12}
\end{tabular}
\end{small}
\medskip

\begin{caption}{Single threaded times in seconds taken on a 5.6GHz Intel i9-13900KS.
The $t_{\rm HCP}$ column is the time to compute the $H_D$ of degree $h$ and logarithmic height $\H$, the $t_{\rm CM}$ columns are times on input $H_D$ and the $t_{\rm noCM}$ columns are times on input $H_D+1$.}\label{tab:timings}
\end{caption}
\end{center}
\end{table}

\FloatBarrier

\bibliographystyle{amsplain}
\providecommand{\bysame}{\leavevmode\hbox to3em{\hrulefill}\thinspace}
\providecommand{\MR}{\relax\ifhmode\unskip\space\fi MR }
\providecommand{\MRhref}[2]{\href{http://www.ams.org/mathscinet-getitem?mr=#1}{#2}}
\providecommand{\href}[2]{#2}

\end{document}